\documentclass[12pt]{article}

\pdfoutput=1

\usepackage[utf8]{inputenc}
\usepackage[T1]{fontenc}
\usepackage[english]{babel}
\usepackage{amsmath}
\usepackage{amsfonts}
\usepackage{amssymb}
\usepackage{amsthm}
\usepackage{mathtools}
\allowdisplaybreaks%
\usepackage{dsfont}
\usepackage{cite}
\usepackage{url}
\usepackage{xcolor}
\usepackage{tikz}
\usepackage[totalwidth=480pt, totalheight=680pt]{geometry}
\usepackage[colorlinks=true, linkcolor=blue]{hyperref}

\DeclareMathOperator{\DIAG}{diag}
\DeclareMathOperator{\DIF}{d\!}
\DeclareMathOperator{\IM}{im}
\DeclareMathOperator{\COL}{col}
\DeclareMathOperator{\KER}{ker}
\DeclareMathOperator*{\MIN}{minimize}

\DeclareMathOperator{\TR}{tr}

\DeclarePairedDelimiter{\card}{\lvert}{\rvert}

\DeclarePairedDelimiterXPP{\twonorm}[1]{\sigma_1}{\lparen}{\rparen}{}{#1}

\DeclarePairedDelimiterXPP{\Fnorm}[1]{}{\lVert}{\rVert}{_F}{#1}
\DeclarePairedDelimiterXPP{\Htwonorm}[1]{}{\lVert}{\rVert}{_{\cH_2}}{#1}
\DeclarePairedDelimiterXPP{\Hinfnorm}[1]{}{\lVert}{\rVert}{_{\cH_\infty}}{#1}
\DeclarePairedDelimiterXPP{\Hpnorm}[1]{}{\lVert}{\rVert}{_{\cH_p}}{#1}
\DeclarePairedDelimiter{\PAREN}{\lparen}{\rparen}
\DeclarePairedDelimiter{\BRACK}{\lbrack}{\rbrack}

\newcommand{\aeptxt}{\mathrm{AEP}}

\newcommand{\R}{\mathbb{R}}

\newcommand{\Rnn}{\R^{n \times{}n}}

\newcommand{\Rpp}{\R^{p \times{}p}}
\newcommand{\Rnr}{\R^{n \times{}r}}
\newcommand{\Rrr}{\R^{r \times{}r}}

\newcommand{\cA}{\mathcal{A}}

\newcommand{\cG}{\mathcal{G}}
\newcommand{\cH}{\mathcal{H}}
\newcommand{\cP}{\mathcal{P}}
\newcommand{\cV}{\mathcal{V}}
\newcommand{\cX}{\mathcal{X}}
\newcommand{\bbone}{\mathds{1}}

\newcommand{\Serr}{\Delta}

\newcommand{\graph}{\cG}
\newcommand{\vlf}{\cV}
\newcommand{\vl}{\vlf_\mathrm{L}}
\newcommand{\vf}{\vlf_\mathrm{F}}
\newcommand{\adjmat}{\cA}

\newcommand{\PTP}{\PAREN[\big]{P^T P}}
\newcommand{\half}{\frac{1}{2}}

\theoremstyle{plain}
\newtheorem{TM}{Theorem}
\newtheorem{PROP}[TM]{Proposition}
\newtheorem{LM}[TM]{Lemma}

\theoremstyle{definition}
\newtheorem{DF}[TM]{Definition}

\theoremstyle{remark}
\newtheorem{RM}[TM]{Remark}

\let\le\leq
\let\ge\geq
\let\preceq\leq
\let\succeq\geq
\let\widehat\hat
\let\cdot\times
\let\subseteq\subset

\begin{document}
\title{Model Reduction of Linear Multi-Agent Systems by Clustering and
Associated $\cH_2$- and $\cH_\infty$-Error Bounds}

\author{Hidde-Jan~Jongsma%
  \thanks{H.-J. Jongsma and H.~L. Trentelman are with the Johann Bernoulli
  Institute for Mathematics and Computer Science, University of Groningen,
  Groningen, The Netherlands. E-mails: \protect\url{{h.jongsma,
  h.l.trentelman}@rug.nl}.}
  \and
  Petar~Mlinari\'{c}%
  \thanks{P. Mlinari\'{c}, S. Grundel, and P. Benner are with the Max Planck
  Institute for Dynamics of Complex Technical Systems, Sandtorstr.\ 1, 39106
  Magdeburg, Germany. E-mails: \protect\url{{mlinaric, grundel,
  benner}@mpi-magdeburg.mpg.de}.
  P. Mlinari\'{c} is also affiliated to the ``International Max Planck Research
  School (IMPRS) for Advanced Methods in Process and System Engineering
  (Magdeburg)''.}
  \and
  Sara~Grundel\footnotemark[2]
  \and
  Peter~Benner\footnotemark[2]
  \and
  Harry~L.~Trentelman\footnotemark[1]
}

\date{}

\maketitle

\begin{abstract}
In this paper, we study a model reduction technique for leader-follower
networked multi-agent systems defined on weighted, undirected graphs with
arbitrary linear multivariable agent dynamics. In the network graph of this
network, nodes represent the agents and edges represent communication links
between the agents.  Only the leaders in the network receive an external input,
the followers only exchange information with their neighbors. The reduced
network is obtained by partitioning the set of nodes into disjoint sets, called
clusters, and associating with each cluster a single, new, node in a reduced
network graph.  The resulting reduced network has a weighted, symmetric,
directed network graph, and inherits some of the structure of the original
network. We establish a priori upper bounds on the $\cH_2$ and $\cH_\infty$
model reduction error for the special case that the graph partition is almost
equitable. These upper bounds depend on the Laplacian eigenvalues of the
original and reduced network, an auxiliary system associated with the agent
dynamics, and the number of nodes that belong to the same clusters as the
leaders in the network. Finally, we consider the problem of obtaining a priori
upper bounds if we cluster using arbitrary, possibly non almost equitable,
partitions.
\end{abstract}


\section{Introduction}

In the last few decades, the world has become increasingly connected. This has
brought a significant interest to fields such as complex networks, smart-grids,
distributed systems, transportation networks, biological networks, and networked
multi-agent systems, see e.g.~\cite{New10, EstFHetal10, BenFFetal14}. Widely
studied problems in networked systems are the problems of consensus and
synchronization, see~\cite{Mor05, OlfM03, LiDCetal10, MaZ10}. In the consensus
problem, the goal is to have the agents in the network reach agreement on
certain physical or measured quantities depending on the states of all the
agents, where the agents use only locally available information. Other important
subjects in the theory of networked systems are flocking, formation control,
sensor placement, and controllability of networks, see e.g.~\cite{FaxM04,
CorMKetal02, Olf06, EgeMCetal12, GaoLDetal14, MonCT15}.

A networked multi-agent system is a network consisting of multiple
interconnected dynamical systems called \emph{agents}. The interconnection
topology of the network is modeled by a graph called the \emph{network graph},
which specifies for each agent its neighbors in the network. In this graph, the
agents are represented by \emph{nodes} and the communication links are
represented by \emph{edges}. In the network, the agents exchange relative state
or output information with their neighbors. Depending on the problem, the
network graph can be weighted or unweighted, and directed or undirected. In the
present paper, the original networks are assumed to have weighted, undirected
network graphs.

Behavioral analysis and controller design for large-scale complex networks can
potentially become extremely expensive from a computational point of view,
especially for problems where the complexity of the network scales as a power of
the number of nodes it contains. In order to tackle this problem, there is a
need for methods and procedures to approximate the original networks by smaller,
less complex ones.

Direct application of established model reduction techniques, such as balanced
truncation, Hankel-norm approximation, and Krylov subspace methods, see
e.g.~\cite{morAnt05, morBenMS05}, to the dynamical models of networked systems
generally leads to a collapse of the network structure, as well as the loss of
important properties such as consensus. Furthermore, the resulting reduced
models often cannot even be interpreted as networked systems anymore.

While there do exist structure-preserving techniques which preserve certain
properties such as the Lagrangian structure~\cite{morLalKM03}, the second order
structure~\cite{morLiB04, morBenKS13}, and the interconnection structure of
interconnected subsystems~\cite{morReiS07, morVanV08, morSanM09}, multi-agent
systems possess their own specific internal structure: the topology of the
network. In the past, model reduction techniques specifically for networked
multi-agent systems have been proposed in~\cite{morImu12, morIshKIetal12a,
morBesSJ14, morMliGB15}. These methods are based on clustering nodes in the
network. With clustering, the idea is to partition the set of nodes in the
network graph into disjoint sets called clusters, and to associate with each
cluster a single, new, node in the reduced network, thus reducing the number of
nodes and connections and the complexity of the network topology.  Other
techniques instead reduce the network topology in a different manner, for
instance by removing connections in the network graph that are of lesser
importance, see e.g.~\cite{morJonTC15}.

In~\cite{morMonTC14} a model reduction technique was introduced that harnesses a
specific class of graph partitions called \emph{almost equitable partitions}.
The results in~\cite{morMonTC14} provide explicit expressions for the $\cH_2$
model reduction error if a \emph{leader-follower network} with single integrator
agent dynamics is clustered according to an almost equitable partition of the
network graph. In a leader-follower network, a subset of the nodes receive an
external input. These nodes are called the \emph{leaders} of the network. The
other nodes only receive relative information from their neighbors in the graph,
these are called the \emph{followers}. In the present paper, we extend the
results in~\cite{morMonTC14} to networks where the agent dynamics is given by an
arbitrary multivariable input-state-output system. We provide a priori upper
bounds on both the $\cH_2$ and the $\cH_\infty$ model reduction errors if the
agents are clustered according to almost equitable partitions. Compared
to~\cite{morMonTC14}, we use a slightly different output equation to measure the
disagreement between the agents in the network, which enables us to also
consider the problem of clustering a network according to arbitrary, not
necessarily almost equitable, graph partitions.

The outline of this paper is as follows. In Section~\ref{sec:prelim} we
introduce some notation and review the theory needed for computing the $\cH_2$
and $\cH_\infty$ model reduction error bounds in the remainder of the paper. In
Section~\ref{sec:problem} we precisely formulate the problem of model reduction
of leader-follower networks with arbitrary agent dynamics.
Section~\ref{sec:partitions} reviews the needed theory on graph partitions and
introduces the reduced network, obtained by applying a Petrov-Galerkin
projection to the dynamical system of the original network. In
Section~\ref{sec:H_2} we provide a priori error bounds on the $\cH_2$ model
reduction error for networks with arbitrary agent dynamics, clustered according
to almost equitable partitions. In Section~\ref{sec:H_inf}, we complement these
results by providing upper bounds on the $\cH_\infty$ model reduction error. In
Section~\ref{sec:general} the problem of clustering networks according to
general partitions is considered and the first steps towards a priori error
bounds on both the $\cH_2$ and $\cH_\infty$ model reduction errors are made.
Finally, Section~\ref{sec:conclusions} provides some conclusions.

\section{Preliminaries}\label{sec:prelim}

The trace of a square matrix $A$ is denoted $\TR(A)$ and is the sum of the
diagonal entries of $A$. For matrices $A$, $B$, and $C$ of appropriate
dimensions such that $A B C$ is square, the trace of $A B C$ satisfies
\begin{equation*}
  \TR(A B C) = \TR(C A B) = \TR(B C A).
\end{equation*}
The largest singular value of a matrix $A$ is denoted $\twonorm{A}$ and
satisfies $\twonorm{A} = \lambda_{\max}{(A^T A)}^{\frac{1}{2}}$. For given real
numbers $\alpha_1, \alpha_2, \ldots, \alpha_k$, let $\DIAG(\alpha_1, \alpha_2,
\ldots, \alpha_k)$ denote the $k \times k$ diagonal matrix with the $\alpha_i$'s
on the diagonal. In the case of a collection of square matrices $A_1, A_2,
\ldots, A_k$, we use $\DIAG(A_1, A_2, \ldots, A_k)$ to denote the block diagonal
matrix with the $A_i$'s as diagonal blocks. For a rectangular matrix $A$, let
$A^+$ denote its Moore-Penrose pseudoinverse.

Consider the input-state-output system
\begin{equation}\label{eq:simple-system}
  \begin{aligned}
    \dot{x} &= A x + B u, \\
    y &= C x,
  \end{aligned}
\end{equation}
with $x \in \R^n$, $u \in \R^m$, $y \in \R^p$, and transfer function $S(s) = C
{(s I - A)}^{-1} B$. If $S$ has all its poles in the open left half complex
plane, then we define its $\cH_2$-norm by
\begin{equation*}
  \Htwonorm{S}^2
  :=
  \frac{1}{2 \pi}
  \int^{+\infty}_{-\infty} \TR\PAREN*{{S(-i \omega)}^T S(i \omega)}
  \DIF{\omega}.
\end{equation*}
It is well known that if $A$ is Hurwitz, then the $\cH_2$-norm can be computed
as
\begin{equation*}
  \Htwonorm{S}^2
  =
  \TR\PAREN*{B^T X B},
\end{equation*}
where $X$ is the unique positive semi-definite solution of the Lyapunov equation
\begin{equation}\label{eq:lyap}
  A^T X + X A + C^T C = 0.
\end{equation}
For the purposes of this paper, we also need to deal with the situation when $A$
is not Hurwitz. Let ${\cX}_+(A)$ denote the generalized unstable subspace of
$A$, i.e., the direct sum of the generalized eigenspaces of $A$ corresponding to
its eigenvalues in the closed right half plane. We state the following proposition:
\begin{PROP}\label{prop:Lyapunov}
  Assume that ${\cX}_+(A) \subset \ker C$. Then the Lyapunov
  equation~\eqref{eq:lyap} has at least one positive semi-definite solution.
  Among all positive semi-definite solutions, there is exactly one solution, say
  $X$, with the property ${\cX}_+(A) \subset \ker X$. For this particular
  solution $X$ we have $\Htwonorm{S}^2 = \TR\PAREN*{B^T X B}$.
\end{PROP}
\begin{proof}
  Without loss of generality, assume that
  \begin{equation*}
    A = \begin{pmatrix}A_{-} & 0 \\ 0 & A_{+}\end{pmatrix}, \quad
    B = \begin{pmatrix}B_{-} \\ B_{+}\end{pmatrix}, \quad
    C = \begin{pmatrix}C_{-} & 0\end{pmatrix},
  \end{equation*}
  where $A_{-}$ is Hurwitz, and $A_{+}$ has all its eigenvalues in the closed
  right half plane. Let $X_{-}$ be the unique solution to the reduced Lyapunov
  equation
  \begin{equation}\label{eq:lyap-reduced}
    A_{-}^T X_{-} + X_{-} A_{-} + C_{-}^T C_{-} = 0.
  \end{equation}
  Then $X_{-} = \int_{0}^{\infty} e^{A_{-}^T t} C_{-}^T C_{-} e^{A_{-} t} \ dt
  \geq 0$. Obviously then, $X = \DIAG(X_{-}, 0)$ is a positive semi-definite
  solution of~\eqref{eq:lyap}. Now let $X$ be a positive semi-definite solution
  to~\eqref{eq:lyap} with the property that $\cX_{+}(A) \subset \ker X$. Then
  $X$ must be of the form $X = \DIAG(X_1, 0)$, and $X_1$ must satisfy the
  reduced Lyapunov equation~\eqref{eq:lyap-reduced}. Thus $X = \DIAG(X_{-}, 0)$.
  Finally, $S$ is stable since $\cX_{+}(A) \subset \ker C$. Moreover,
  \begin{align*}
    \Htwonorm{S}^2 &= \TR \PAREN*{B^T \int_0^\infty e^{A^T t} C^T C e^{A t} \DIF{t} \ B} \\
                 &= \TR \PAREN*{B_{-}^T \int_0^\infty e^{A_{-}^T t} C_{-}^T C_{-} e^{A_{-} t} \DIF{t} \ B_{-}} \\
                 &= \TR \PAREN*{B_{-}^T X_{-} B_{-}} \\
                 &= \TR \PAREN*{B^T X B}.
  \end{align*}
\end{proof}
If $S$ has all its poles in the open left half plane, then the $\cH_\infty$-norm
of $S$ is defined by
\begin{equation*}
  \Hinfnorm{S} := \sup_{\omega \in \R} \twonorm{S(i \omega)}.
\end{equation*}
We will now deal with computing the $\cH_{\infty}$-norm of a stable transfer
function. The result is a generalization of Lemma~4 in~\cite{morIshKIetal14}.
\begin{LM}\label{LM:suff_cond}
  Consider the system~\eqref{eq:simple-system}. Assume that its transfer
  function $S$ has all its poles only in the open left half plane. If there
  exists $X \in \Rpp$ such that $X = X^T$ and $C A = X C$, then $\Hinfnorm{S} =
  \twonorm{S(0)}$.
\end{LM}
\begin{proof}
  For the first part of the proof, let us assume that $(A, B, C)$ is minimal.
  Then, in particular, $A$ is a Hurwitz matrix and $(A, B)$ is controllable.

  Clearly, the inequality $\Hinfnorm{S} \ge \twonorm{S(0)}$ is always satisfied.
  We will prove that $\Hinfnorm{S} \le \twonorm{S(0)}$ using the Bounded Real
  Lemma~\cite{Ran96}, which states that $\Hinfnorm{S} \le \gamma$ if and only if
  there exists $P \in \Rnn$ such that $P = P^T$ and
  \begin{align*}
    A^T P + P A + C^T C + \frac{1}{\gamma^2} P B B^T P \preceq 0.
  \end{align*}
  Let us take $\gamma = \twonorm{S(0)} = \twonorm{C A^{-1} B}$. This implies
  that
  \begin{align}
    C A^{-1} B B^T A^{-T} C^T \preceq \gamma^2 I_p. \label{eq:gamma}
  \end{align}
  Defining $P := -A^{-T} C^T X C A^{-1}$ and using~\eqref{eq:gamma} gives us
  \begin{align*}
    A^T P + P A & + C^T C + \frac{1}{\gamma^2} P B B^T P \\
                & = -C^T X C A^{-1} - A^{-T} C^T X C + C^T C \\
                & \qquad + \frac{1}{\gamma^2} A^{-T} C^T X C A^{-1} B B^T A^{-T}
                  C^T X C A^{-1} \\
                & \leq -C^T X C A^{-1} - A^{-T} C^T X C + C^T C
                  + A^{-T} C^T X X CA^{-1} \\
                & = \PAREN{X C A^{-1} - C}^T \PAREN{X C A^{-1} - C} \\
                & = 0.
  \end{align*}
  From the Bounded Real Lemma, we conclude that $\Hinfnorm{S} \le
  \twonorm{S(0)}$.

  For a non-minimal representation $(A, B, C)$, applying the Kalman
  decomposition, let $T$ be a nonsingular matrix such that
  \begin{align*}
    T^{-1} A T & =
    \begin{pmatrix}
      A_1 & 0 & A_6 & 0 \\
      A_2 & A_3 & A_4 & A_5 \\
      0 & 0 & A_7 & 0 \\
      0 & 0 & A_8 & A_9
    \end{pmatrix},
    \quad T^{-1} B =
    \begin{pmatrix}
      B_1 \\
      B_2 \\
      0 \\
      0
    \end{pmatrix},
    \quad
    C T =
    \begin{pmatrix}
      C_1 & 0 & C_2 & 0
    \end{pmatrix},
  \end{align*}
  where $(A_1, B_1, C_1)$ is a minimal representation of $(A, B, C)$ with $A_1$
  Hurwitz. Obviously,
  \begin{align*}
    (C T) \PAREN{T^{-1} A T}
    = C A T
    = X C T
    = X (C T),
  \end{align*}
  thus the condition is preserved under system transformation. From this, it
  follows that $C_1 A_1 = X C_1$. Therefore, the minimal representation satisfies the
  sufficient condition and using the result obtained above the proof is
  completed.
\end{proof}

Continuing our effort to compute the $\cH_{\infty}$-norm, we formulate a lemma
that will be instrumental in evaluating a transfer function at the origin.
Recall that for a given matrix $A$, its Moore-Penrose inverse is denoted by
$A^+$.
\begin{LM}\label{LM:H0}
  Consider the system~\eqref{eq:simple-system}. If $A$ is symmetric and
  $\KER{A} \subseteq \KER{C}$, then $0$ is not a pole of the transfer function
  $S$ and we have $S(0) = -C A^+ B$.
\end{LM}
\begin{proof}
  If $A$ is nonsingular, then the conclusion follows immediately. Otherwise, let
  $A = U \Lambda U^T$ be an eigenvalue decomposition with orthogonal $U$ and
  $\Lambda = \DIAG(0, \Lambda_2)$, where $\Lambda_2 \in \Rrr$ and $r$ is the
  rank of $A$. We denote $U = \begin{pmatrix} U_1 & U_2 \end{pmatrix}$, with
  $U_2 \in \Rnr$. Then
  \begin{align*}
    A^+
    = U \Lambda^+ U^T
    =
    \begin{pmatrix}
      U_1 & U_2
    \end{pmatrix}
    \begin{pmatrix}
      0 & 0 \\
      0 & \Lambda_2^{-1}
    \end{pmatrix}
    \begin{pmatrix}
      U_1^T \\
      U_2^T
    \end{pmatrix}
    =
    U_2 \Lambda_2^{-1} U_2^T.
  \end{align*}
  Note that $CU_1 = 0$. We have
  \begin{align*}
    S(s)
    & =
    C U {(s I - \Lambda)}^{-1} U^T B \\
    & =
    C
    \begin{pmatrix}
      U_1 & U_2
    \end{pmatrix}
    \begin{pmatrix}
      s^{-1} I & 0 \\
      0 & {(s I- \Lambda_2)}^{-1}
    \end{pmatrix}
    \begin{pmatrix}
      U_1^T \\
      U_2^T
    \end{pmatrix}
    B \\
    & =
    C U_2 {(s I - \Lambda_2)}^{-1} U_2^T B.
  \end{align*}
  Hence, $S(s)$ is defined at $s = 0$ and $S(0) = -C U_2 \Lambda_2^{-1} U_2^T B
  = -C A^+ B$.
\end{proof}

Finally we discuss the model reduction technique known as Petrov-Galerkin
projection.
\begin{DF}
  Consider the system~\eqref{eq:simple-system}. Let $W, V \in \Rnr$, with $r <
  n$, such that $W^T V = I$. The matrix $V W^T$ is then a projector, called a
  \emph{Petrov-Galerkin projector}. The reduced order system
  \begin{align*}
    \dot{\hat x} &= W^T A V \hat x + W^T B u, \\
    \hat y &= C V \hat x,
  \end{align*}
  with $\hat x \in \R^r$ is called \emph{the Petrov-Galerkin projection} of the
  original system~\eqref{eq:simple-system}.
\end{DF}

\section{Problem formulation}\label{sec:problem}

In this paper, we consider networks of diffusively coupled linear subsystems.
These subsystems, called \emph{agents}, have identical dynamics, however a
selected subset of the agents, called the \emph{leaders}, also receives an input
from outside the network. The remaining agents are called \emph{followers}. The
network consists of $N$ agents, indexed by $i$, so $i \in \vlf := \{1, 2,
\ldots, N\}$. The subset $\vl \subset \vlf$ is the index set of the leaders,
more explicitly $\vl = \{v_1, v_2, \ldots, v_m\}$. The followers are indexed by
$\vf : = \vlf \setminus \vl$. More specifically, the leaders are represented by
the finite dimensional linear system
\begin{equation*}
  \dot x_i = A x_i + B \sum_{j = 1}^N a_{ij} (x_j - x_i) + E u_j, ~~i \in \vl,\
  i = v_j,
\end{equation*}
whereas the followers have dynamics
\begin{equation*}
  \dot x_i = A x_i + B \sum_{j = 1}^N a_{ij} (x_j - x_i), ~~i \in \vf.
\end{equation*}
The weights $a_{ij} \geq 0$ represent the coupling strengths of the diffusive
coupling between the agents. In this paper, we assume that $a_{ij} = a_{ji}$ for
all $i, j \in \vlf$. Also, $a_{ii} = 0$ for all $i \in \vlf$. Furthermore, $x_i
\in \R^n$ is the state of agent $i$, and $u_j \in \R^r$ is the external input to
the leader $v_j$. Finally, $A \in \R^{n \times n}$, $B\in \R^{n \times n}$ and
$E \in \R^{n \times r}$ are real matrices. It is customary to represent the
interaction between the agents by the graph $\graph$ with node set $\vlf = \{1,
2, \ldots, N\}$ and adjacency matrix $\adjmat = (a_{ij})$. In the set up of this
paper, this graph is undirected, reflecting the assumption that $\adjmat$ is
symmetric. The \emph{Laplacian matrix} of the graph $\graph$ is denoted by $L$
and defined as
\begin{equation*}
  L_{ij}
  =
  \begin{cases}
    d_i & \text{if } i = j, \\
    -a_{ij} & \text{if } i \neq j.
  \end{cases}
\end{equation*}
with $d_i = \sum_{j=1}^N a_{ij}$.

Recall that the set of leader nodes is $\vl = \{v_1, v_2, \ldots, v_m\}$, and
define the matrix $M \in \R^{N \times m}$ as
\begin{equation*}
  M_{ij}
  =
  \begin{cases}
    1 & \text{if } i = v_j, \\
    0 & \text{otherwise}.
  \end{cases}
\end{equation*}
Denote $x = \COL(x_1, x_2, \ldots, x_N)$ and $u = \COL(u_1, u_2, \ldots, u_m)$.
The total network is then represented by
\begin{equation}\label{eq:original-system}
  \dot x = (I \otimes A - L \otimes B) x + (M \otimes E) u.
\end{equation}
The goal of this paper is to find a reduced order networked system, whose
dynamics is a good approximation of the networked
system~\eqref{eq:original-system}. In this paper, the idea to obtain such
approximation is to {\em cluster\/} groups of agents in the network, and to
treat each of the resulting clusters as a node in a new, reduced order, network.
The reduced order network will again be a leader-follower network, and by the
clustering procedure essential interconnection features of the network will be
preserved. We will require that the {\em synchronization\/} properties of the
network are preserved after reduction. We will assume that the original network
is synchronized, meaning that if the external inputs $u_j = 0$ for $j=1, 2,
\ldots, m$, the network reaches synchronization, that is, for all $i,j \in
\vlf$, we have
\begin{equation*}
  x_i(t) - x_j(t) \rightarrow 0
\end{equation*}
as $t \rightarrow \infty$. We will impose that the reduction procedure preserves
this property. In this paper, a standing assumption will be that the graph
$\graph$ of the original network is {\em connected}. This is equivalent to the
condition that $0$ is a simple eigenvalue of the Laplacian $L$,
see~\cite[Theorem~2.8]{MesE10}. In this case, the network reaches
synchronization if and only if $(L \otimes I) x(t) \rightarrow 0$ as $t
\rightarrow \infty$.

In order to be able to compare the original network~\eqref{eq:original-system}
with its reduced order approximation and to make statements about the
approximation error, we need a notion of \emph{distance} between the networks.
One way to obtain such notion is to introduce an \emph{output} associated with
the network~\eqref{eq:original-system}. By doing this, both the original network
and its approximation become input-output systems, and we can compare them by
looking at the difference of their transfer functions. Being a measure for the
disagreement between the states of the agents in~\eqref{eq:original-system}, we
choose $y = (L \otimes I) x$ as the output of the original network. Indeed, this
output $y$ can be considered a measure of the disagreement in the network, in
the sense that $y(t)$ is small if and only if the network is close to being
synchronized. Thus, with the original system~\eqref{eq:original-system} we now
identify the input-state-output system:
\begin{equation}\label{eq:extended-system}
  \begin{aligned}
    \dot x &= (I \otimes A - L \otimes B) x + (M \otimes E) u, \\
    y &= (L \otimes I) x.
  \end{aligned}
\end{equation}
The state space dimension of~\eqref{eq:extended-system} is equal to $n N$, its
number of inputs equals to $m r$, and the number of outputs is $n N$.

In this paper, we will use clustering to obtain a reduced order network, i.e.\ a
network with a reduced number of agents, as an approximation of the original
network~\eqref{eq:extended-system}. We also aim at deriving upper bounds for the
approximation error. We will obtain upper bounds both for the ${\cH}_2$-norm as
well as the $\cH_{\infty}$-norm of the difference of the transfer functions of
the original network and its approximation.

\section{Graph partitions and reduction by clustering}\label{sec:partitions}

We consider networks whose interaction topologies are represented by weighted
graphs $\graph$ with node set $\vlf$. The graph of the original
network~\eqref{eq:original-system} is undirected, however, our reduction
procedure will lead to networks on directed graphs. As before, the adjacency
matrix of the graph $\graph$ is the matrix $\adjmat = (a_{ij})$, where $a_{ij}
\geq 0$ is the weight of the arc from node $j$ to node $i$. As noted before, the
graph is undirected if and only if $\adjmat$ is symmetric.

A nonempty subset $C \subseteq \vlf$ is called a \emph{cell} or \emph{cluster}
of $\vlf$. A \emph{partition} of a graph is defined as follows.
\begin{DF}
  Let $\graph$ be an undirected graph. A partition $\pi = \{C_1, C_2, \ldots,
  C_k\}$ of $\vlf$ is a collection of cells such that $\vlf = \bigcup_{i = 1}^k
  C_i$ and $C_i \cap C_j = \emptyset$ whenever $i \neq j$. When we say that
  $\pi$ is a partition of $\graph$, we mean that $\pi$ is a partition of the
  vertex set $\vlf$ of $\graph$. Nodes $i$ and $j$ are called \emph{cellmates}
  in $\pi$ if they belong to the same cell of $\pi$. The \emph{characteristic
  vector} of a cell $C \subseteq \vlf$ is the $N$-dimensional column vector
  $p(C)$ defined as
  \begin{equation*}
    p_i(C)
    =
    \begin{cases}
      1 & \text{if } i \in C, \\
      0 & \text{otherwise.}
    \end{cases}
  \end{equation*}
  The \emph{characteristic matrix} of the partition $\pi = \{C_1, C_2, \ldots,
  C_k\}$ is defined as the $N \times k$ matrix
  \begin{equation*}
    P(\pi)
    =
    \begin{pmatrix}
      p(C_1) & p(C_2) & \cdots & p(C_k)
    \end{pmatrix}.
  \end{equation*}
\end{DF}
For a given partition $\pi = \{C_1, C_2, \ldots, C_k\}$, consider the cells
$C_p$ and $C_q$ with $p \neq q$. For any given node $j \in C_q$, we define its
\emph{degree with respect to $C_p$} as the sum the weights of all arcs from $j$
to $i \in C_p$, i.e.\ the number
\[
  d_{pq}(j) := \sum_{i \in C_p} a_{ij}.
\]
Next, we will construct a reduced order approximation
of~\eqref{eq:extended-system} by clustering the agents in the network according
to a partition of $\graph$. Let $\pi$ be a partition of $\graph$, and let $P :=
P(\pi)$ be its characteristic matrix. Extending the main idea
in~\cite{morMonTC14}, we take as reduced order system the Petrov-Galerkin
projection of the original system~\eqref{eq:extended-system}, with the following
choice for the matrices $V$ and $W$:
\begin{equation*}
  W = P \PTP^{-1} \otimes I, \quad
  V = P \otimes I.
\end{equation*}
The dynamics of the resulting reduced order model is then given by
\begin{equation}\label{eq:reduced-system}
  \begin{aligned}
    \dot{\widehat{x}} &= (I \otimes A - \widehat{L} \otimes B) \widehat{x} +
    (\widehat{M} \otimes E) u \\
    \widehat{y} &= (L P \otimes I) \widehat{x}.
  \end{aligned}
\end{equation}
where
\begin{equation*}
  \begin{aligned}
    \widehat{L} &= \PTP^{-1} P^T L P \\
    \widehat{M} &= \PTP^{-1} P^T M,
  \end{aligned}
\end{equation*}
We claim that the matrix $\hat L$ is the Laplacian of a weighted \emph{directed}
graph with node set $\{1, 2, \ldots, k\}$, with $k$ equal to the number of
clusters in the partition $\pi$. Indeed, by inspection it can be seen that the
adjacency matrix of this reduced graph is $\hat{\adjmat} = (\hat{a}_{pq})$, with
\[
  \hat{a}_{pq} = \frac{1}{\card{C_p}} \sum_{j \in C_q} d_{pq}(j),
\]
where $d_{pq}(j)$ is the degree of $j \in C_q$ with respect to $C_p$, and
$\card{C_p}$ the cardinality of $C_p$. Note also that the row sums of $\hat L$
are equal to zero since $\hat L \bbone_k = 0$. The matrix $\hat M \in \R^{k
\times m}$ satisfies
\begin{equation*}
  \hat{M}_{pj}
  =
  \begin{cases}
    \frac{1}{\card{C_p}} & \text{if } v_j \in C_p, \\
    0 & \text{otherwise},
  \end{cases}
\end{equation*}
where $v_1, v_2, \ldots, v_m$ are the leader nodes, $p = 1, 2, \ldots, k$,
and $j = 1, 2, \ldots, m$.

Clearly, the state space dimension of the reduced order
network~\eqref{eq:reduced-system} is equal to $n k$, whereas the dimensions $m
r$ and $n N$ of the input and output have remained unchanged. Thus we can
investigate the error between the original and reduced order network by looking
at the difference of their transfer functions. In the sequel we will both
investigate the ${\cH}_2$-norm as well as the ${\cH}_{\infty}$-norm of this
difference.

Before doing this however, we will now first study the question whether our
reduction procedure preserves synchronization. It is important to note that
since, by assumption, the original undirected graph is connected, it has a
directed spanning tree. It is easily verified that this property is preserved by
our clustering procedure. Then, since the property of having a directed spanning
tree is equivalent with 0 being a simple eigenvalue of the Laplacian
(see~\cite[Proposition 3.8]{MesE10}), the reduced order Laplacian $\hat{L}$
has again 0 as a simple eigenvalue.

Now assume that the original network~\eqref{eq:extended-system} is synchronized.
It is well known, see e.g.~\cite{TreTM13}, that this is equivalent with the
condition that for each nonzero eigenvalue $\lambda$ of the Laplacian $L$ the
matrix $A - \lambda B$ is Hurwitz. Thus, synchronization is preserved if and
only if for each nonzero eigenvalue $\hat{\lambda}$ of the reduced order
Laplacian $\hat{L}$ the matrix $A - \hat{\lambda}B$ is Hurwitz.

Unfortunately, in general $A - \lambda B$ Hurwitz for all nonzero $\lambda \in
\sigma(L)$ does \emph{not} imply that $A - \hat{\lambda} B$ Hurwitz for all
nonzero $\lambda \in \sigma(\hat{L})$. An exception is the `single integrator'
case $A = 0$ and $B = 1$, where this condition is trivially satisfied, so in
this special case synchronization is preserved. Also if we restrict ourselves to
a special type graph partitions, namely \emph{almost equitable partitions}, then synchronization
turns out to be preserved. We will review this type of partition now.

Again let $\graph$ be a weighted, undirected graph, and let $\pi = \{C_1, C_2,
\ldots, C_k\}$ be a partition of $\graph$. Given two clusters $C_p$ and $C_q$
with $p \neq q$, and a given node $j \in C_q$, recall that $d_{pq}(j)$ denotes
its degree with respect to $C_p$. We call the partition $\pi$ an \emph{almost equitable partition (AEP)}
if for each $p, q$ with $p \neq q$, the degree $d_{pq}(j)$ is independent of $j
\in C_q$, i.e.\ $d_{pq}(j_1) = d_{pq}(j_2)$ for all $j_1, j_2 \in C_q$.

It is a well known fact (see~\cite{CarDR07}) that $\pi$ is an AEP if and
only if the image of its characteristic matrix in invariant under the Laplacian.
\begin{LM}\label{lem:aep}
  Consider the weighed undirected graph $\graph$ with Laplacian matrix $L$. Let
  $\pi$ be a partition of $\graph$ with characteristic matrix $P := P(\pi)$.
  Then $\pi$ is an almost equitable partition if and only if $L \IM P \subset \IM P$.
\end{LM}
As an immediate consequence, the reduced Laplacian $\hat{L}$ obtained using
an AEP satisfies $L P = P \hat{L}$. Indeed, since $\IM P$ is $L$-invariant we
have $L P = P X$ for some matrix $X$. Obviously we must then have $X = \PTP^{-1}
P^T L P = \hat L$. From this, it follows that $\sigma(\hat L) \subset
\sigma(L)$. It then readily follows that synchronization is preserved if we
cluster according to an AEP:
\begin{TM}
  Assume that the network~\eqref{eq:extended-system} is synchronized. Let $\pi$
  be an almost equitable partition. Then the reduced order network~\eqref{eq:reduced-system}
  obtained by clustering according to $\pi$ is synchronized.
\end{TM}

\section{\texorpdfstring{$\boldsymbol{\cH_2}$}{H2}-error bounds}\label{sec:H_2}

In this section, we investigate the $\cH_2$-norm of the error system mapping the
input $u$ to the difference $y - \hat y$ in the case that the original network
is clustered according to an AEP. Let $S$ and $\hat{S}$ denote the transfer
functions of the original~\eqref{eq:extended-system} and reduced order
network~\eqref{eq:reduced-system}, respectively. We have the following lemma:
\begin{LM}\label{lem:error-orthogonal}
  Let $\pi$ be an almost equitable partition of the graph $\graph$. The approximation error when
  clustering $\graph$ according to $\pi$ then satisfies
  \begin{equation*}
    \Htwonorm[\big]{S - \widehat{S}}^2 =
    \Htwonorm{S}^2 - \Htwonorm[\big]{\widehat{S}}^2.
  \end{equation*}
\end{LM}
\begin{proof}
  First, note that the columns of $P(\pi)$ are orthogonal. We construct a matrix
  $T = \begin{pmatrix} P & Q \end{pmatrix}$, where $P := P(\pi)$, and where the
  $N \times (N - k)$ matrix $Q$ is chosen such that the columns of $T$ form an
  orthogonal basis for $\R^N$. In this case, we have $P^T Q = 0$. Next, we
  apply the state space transformation $x = T \tilde x$ to
  system~\eqref{eq:extended-system}. We obtain
  \begin{equation}\label{eq:transformed-system}
    \begin{aligned}
      \begin{pmatrix}
        \dot {\tilde x}_1 \\
        \dot {\tilde x}_2
      \end{pmatrix}
      &=
      A_e
      \begin{pmatrix}
        {\tilde x}_1 \\
        {\tilde x}_2
      \end{pmatrix}
      +
      B_e
      u \\
      y &=
      C_e
      \begin{pmatrix}
        {\tilde x}_1 \\
        {\tilde x}_2
      \end{pmatrix},
    \end{aligned}
  \end{equation}
  where the matrices $A_e$, $B_e$, and $C_e$ are given by
  \begin{gather*}
    A_e =
    \begin{pmatrix}
      I \otimes A - {\PTP}^{-1} P^T L P \otimes B &
      - {\PTP}^{-1} P^T L Q \otimes B \\
      - \PAREN[\big]{Q^T Q}^{-1} Q^T L P \otimes B &
      I \otimes A - \PAREN[\big]{Q^T Q}^{-1} Q^T L Q \otimes B \\
    \end{pmatrix}, \\
    B_e =
    \begin{pmatrix}
      {\PTP}^{-1} P^T M \otimes E \\
      \PAREN[\big]{Q^T Q}^{-1} Q^T M \otimes E \\
    \end{pmatrix}, \quad
    C_e =
    \begin{pmatrix}
      L P \otimes I &
      L Q \otimes I
    \end{pmatrix}.
  \end{gather*}
  Obviously, in~\eqref{eq:transformed-system} the transfer function from $u$ to
  $y$ is equal to $S$. Furthermore, if the state component ${\tilde x}_2$ is
  truncated from~\eqref{eq:transformed-system}, what we are left with is the
  reduced order model~\eqref{eq:reduced-system}. Since $\pi$ is an AEP of
  $\graph$, by Lemma~\ref{lem:aep}, $\IM P$ is invariant under $L$. From this,
  it follows that not only $Q^T P = 0$, but also
  \begin{equation}\label{eq:orthogonal}
    Q^T L P = 0 \text{ and } Q^T L^2 P = 0.
  \end{equation}
  It is easily checked that
  \begin{equation*}
    S(s) = \hat S(s) + \Delta(s),
  \end{equation*}
  where $\Delta(s)$ is given by
  \begin{equation}\label{eq:delta}
    \begin{split}
      \Delta(s) =
      (L Q \otimes I)
      \PAREN*{s I - \PAREN*{I \otimes A - \PAREN[\big]{Q^T Q}^{-1} Q^T L Q
      \otimes B}}^{-1} \\
      \times \PAREN*{\PAREN[\big]{Q^T Q}^{-1} Q^T M \otimes E}.
    \end{split}
  \end{equation}
  From~\eqref{eq:orthogonal} and~\eqref{eq:delta}, we have ${\hat S}{(-s)}^T
  \Delta(s) = 0$. Thus we find that
  \begin{equation*}
    \Htwonorm{S}^2 = \Htwonorm[\big]{\widehat{S}}^2 + \Htwonorm{\Delta}^2,
  \end{equation*}
  which concludes the proof.
\end{proof}

We will now formulate the main theorem of this section, which establishes an a
priori upper bound for the ${\cH}_2$-norm of the approximation error in the case
that we cluster according to an AEP. Before formulating the theorem, we
discuss some important ingredients. An important role is played by the $N - 1$
auxiliary input-state-output systems
\begin{equation}\label{eq:aux-system}
  \begin{aligned}
    \dot x &= (A - \lambda B) x + E d, \\
    z &= \lambda x,
  \end{aligned}
\end{equation}
where $\lambda$ ranges over the nonzero eigenvalues of the Laplacian $L$. Let
$S_{\lambda}(s) = \lambda {(sI - A + \lambda B)}^{-1} E$ be the transfer
functions of these systems. We assume that the original
network~\eqref{eq:extended-system} is synchronized, so that all of the $A -
\lambda B$ are Hurwitz. Let $\Htwonorm{S_{\lambda}}$ denote the $\cH_2$-norm of
$S_{\lambda}$. Recall that the set of leader nodes is $\vl = \{v_1, v_2,
\ldots, v_m\}$. Node $v_i$ will be called leader $i$. This leader is an element
of cluster $C_{k_i}$ for some $k_i \in \{1, 2, \ldots, k\}$. We now have the
following theorem:
\begin{TM}\label{thm:main}
  Assume that the network~\eqref{eq:extended-system} is synchronized. Let $\pi$
  be an almost equitable partition of the graph $\graph$. The absolute approximation error when
  clustering $\graph$ according to $\pi$ then satisfies
  \begin{equation*}
    \Htwonorm[\big]{S - \widehat{S}}^2 \leq
    S_{\max, \cH_2}^2
    \sum^m_{i = 1} \PAREN*{1 - \frac{1}{\card{C_{k_i}}}},
  \end{equation*}
  where $C_{k_i}$ is the set of cellmates of leader $i$, and
  \[
    S_{\max, \cH_2} :=
    \max_{\lambda \in \sigma(L) \setminus \sigma(\hat L)} \Htwonorm{S_\lambda}.
  \]
  Furthermore, the relative approximation error satisfies
  \begin{equation*}
    \frac{\Htwonorm[\big]{S - \hat S}^2}{\Htwonorm{S}^2} \leq
    \frac{S_{\max, \cH_2}^2}{S_{\min, \cH_2}^2}
    \frac{\sum^m_{i = 1} \PAREN*{1 - \frac{1}{\card{C_{k_i}}}}}{m \PAREN*{1 -
    \frac{1}{N}}},
  \end{equation*}
  where
  \[
    S_{\min, \cH_2} :=
    \min_{\lambda \in \sigma(L) \setminus \{0\}} \Htwonorm{S_\lambda}.
  \]
\end{TM}
\begin{RM}
  We see that with a fixed number of agents and a fixed number of leaders, the
  approximation error is equal to 0 if in each cluster that contains a leader,
  the leader is the only node in that cluster. In general, the upper bound
  increases if the number of cellmates of the leaders increases.
\end{RM}
\begin{proof}
  Recall that $\sigma(\hat L) \subset \sigma(L)$. Label the eigenvalues of $L$
  as $0, \lambda_2, \lambda_3, \ldots, \lambda_N$ in such a way that $0,
  \lambda_2, \lambda_3, \ldots, \lambda_k$ are the eigenvalues of $\hat{L}$.
  Also, without loss of generality, we assume that $\pi$ is \emph{regularly
  formed}, i.e.\ all ones in each of the columns of $P(\pi)$ are consecutive.
  One can always relabel the agents in the graph in such a way that this is
  achieved. For simplicity, we again denote $P(\pi)$ by $P$. Recall that the
  reduced Laplacian matrix is given by $\hat L = {\PTP}^{-1} P^T L P$. From
  Lemma~\ref{lem:error-orthogonal} we have that the approximation error
  satisfies
  \begin{equation*}
    \Htwonorm{S - \hat S}^2 = \Htwonorm{S}^2 - \Htwonorm{\hat S}^2.
  \end{equation*}
  We will first compute the $\cH_2$-norms of $S$ and $\hat S$ separately and
  then give an upper bound for the difference.

  Consider the symmetric matrix
  \begin{equation}\label{eq:L-bar}
    \bar L
    :=
    {\PTP}^{\half} \hat L {\PTP}^{-\half}
    =
    {\PTP}^{-\half} P^T L P{\PTP}^{-\half}.
  \end{equation}
  Note that the eigenvalues of $\bar L$ and $\hat L$ coincide. Let $\hat U$ be
  an orthogonal matrix that diagonalizes $\bar L$. We then have
  \begin{equation}\label{eq:u-hat}
    \hat U^T {\PTP}^{-\half} P^T L P{\PTP}^{-\half} \hat U
    =
    \DIAG(0, \lambda_2, \ldots, \lambda_k)
    =:
    \hat \Lambda.
  \end{equation}
  Next, take $U_1 = P {\PTP}^{-\half} \hat U$. The columns of $U_1$ have unit
  length and are orthogonal:
  \begin{equation*}
    U_1^T U_1
    =
    \hat U^T {\PTP}^{-\half} P^T
    P {\PTP}^{-\half} \hat U
    = \hat U^T \hat U
    = I.
  \end{equation*}
  Furthermore, we have that
  \begin{equation*}
    U_1^T L U_1
    =
    \hat U^T {\PTP}^{-\half} P^T L P{\PTP}^{-\half} \hat U
    =
    \hat \Lambda.
  \end{equation*}
  Now choose $U_2$ such that $U = \begin{pmatrix} U_1 & U_2 \end{pmatrix}$ is an
  orthogonal matrix and
  \begin{equation*}
    \Lambda
    :=
    U^T L U
    =
    \begin{pmatrix}
      \hat \Lambda & 0 \\
      0 & \bar \Lambda
    \end{pmatrix},
  \end{equation*}
  where $\bar \Lambda = \DIAG(\lambda_{k + 1}, \ldots, \lambda_N)$. It is easily
  verified that the first column of $U_1$, and thus the first column of $U$, is
  given by $\frac{1}{\sqrt{N}} \bbone_N$, where $\bbone_N$ is the $N$-vector of
  1's, a fact that we will use in the remainder of this paper. To compute the
  $\cH_2$-norm of $S$ we can use the result of Proposition~\ref{prop:Lyapunov}.
  It can be verified, using the fact that $A - \lambda_i B$ is Hurwitz for $i =
  2, 3 \ldots, N$, that
  \[
    {\cX}_+(I \otimes A - L \otimes B) = \bbone_N \otimes {\cX}_+(A).
  \]
  This immediately implies that ${\cX}_+(I \otimes A - L \otimes B) \subset
  \ker(L \otimes I)$. As a consequence, we have
  \begin{equation*}
    \Htwonorm{S}^2
    =
    \TR\PAREN*{\PAREN*{M^T \otimes E^T} X (M \otimes E)},
  \end{equation*}
  where $X$ is the unique positive semi-definite solution to the Lyapunov
  equation
  \begin{equation}\label{eq:x}
    \PAREN*{I \otimes A^T - L \otimes B^T} X
    + X (I \otimes A - L \otimes B)
    + L^2 \otimes I
    = 0
  \end{equation}
  with the property that ${\cX}_+(I \otimes A - L \otimes B) \subset \ker X$.
  In order to compute this solution $X$, premultiply~\eqref{eq:x} by $U^T
  \otimes I$ and postmultiply by $U \otimes I$, and substitute $Z = (U^T \otimes
  I) X (U \otimes I)$ to obtain
  \begin{equation}\label{eq:z}
    \PAREN*{I \otimes A^T - \Lambda \otimes B^T} Z
    + Z (I \otimes A - \Lambda \otimes B)
    + \Lambda^2 \otimes I
    = 0.
  \end{equation}
  Solving~\eqref{eq:z} we take $Z$ as
  \begin{equation*}
    Z
    =
    \DIAG(0, X_2, \ldots, X_N),
  \end{equation*}
  where $X_i$, for $i = 2, \ldots, N$, is the observability Gramian of the
  auxiliary system $(A - \lambda_i B, E, \lambda_i I)$ in~\eqref{eq:aux-system}.
  Next, $X := (U \otimes I) Z (U^T \otimes I)$ is a solution of the original
  Lyapunov equation, and it is easily verified that indeed ${\cX}_+(I \otimes A
  - L \otimes B) \subset \ker X$. Thus we obtain the following expression for
  the $\cH_2$-norm of $S$:
  \begin{equation}\label{eq:s-norm-identity}
    \begin{aligned}
      \Htwonorm{S}^2
      &=
      \TR\PAREN*{\PAREN*{M^T U \otimes E^T}
      \DIAG(0, X_2, \ldots, X_N) \PAREN*{U^T M \otimes E}}, \\
      &=
      \TR\PAREN*{\PAREN*{U^T M M^T U \otimes I}
      \DIAG(0, E^T X_2 E, \ldots, E^T X_N E)}.
    \end{aligned}
  \end{equation}
  Next, we compute the $\cH_2$-norm for the reduced system. Firstly, it can be
  verified that
  \[
    {\cX}_+(I \otimes A - \hat L \otimes B) = \bbone_k \otimes {\cX}_+(A)
  \]
  This implies that ${\cX}_+(I \otimes A - \hat L \otimes B) \subset \ker(LP
  \otimes I)$. By Proposition~\ref{prop:Lyapunov} we then have
  \begin{equation*}
    \Htwonorm[\big]{\hat S}^2
    =
    \TR\PAREN[\big]{\PAREN[\big]{\widehat{M}^T \otimes E^T} \widehat{X}
    \PAREN[\big]{\widehat{M} \otimes E}},
  \end{equation*}
  where $\hat X$ is the unique positive semi-definite solution to the Lyapunov
  equation
  \begin{equation}\label{eq:x1}
    \PAREN[\big]{I \otimes A^T - \hat L^T \otimes B^T} \hat X
    + \hat X (I \otimes A - \hat L \otimes B)
    + P^T L^2 P \otimes I
    = 0.
  \end{equation}
  with the property that ${\cX}_+(I \otimes A - \hat L \otimes B) \subset \ker
  \hat X$. In order to compute this solution, pre- and
  postmultiply~\eqref{eq:x1} by ${\PTP}^{-\half} \otimes I$ and substitute
  \begin{equation*}
    \widehat{Y}
    =
    \PAREN*{\PTP^{-\half} \otimes I} \widehat{X}
    \PAREN*{\PTP^{-\half} \otimes I}
  \end{equation*}
  to obtain
  \begin{equation}\label{eq:y}
    \begin{split}
      \PAREN*{I \otimes A^T - \bar L \otimes B^T} \widehat{Y}
      + \widehat{Y} \PAREN*{I \otimes A - \bar L \otimes B} \\
      \quad {} + \PTP^{-\half} P^T L^2 P \PTP^{-\half}
      \otimes I
      = 0.
    \end{split}
  \end{equation}
  Recall from Section~\ref{sec:partitions} that $L P = P \hat L$. From this it
  follows that
  \[
    {\PTP}^{-\half} P^T L^2 P {(P^T P)}^{-\half} = \bar L^2.
  \]
  Consequently, we can diagonalize the corresponding term in~\eqref{eq:y} by
  premultiplying by $\hat U^T \otimes I$ and postmultiplying by $\hat U \otimes
  I$, where $\hat U$ is as in~\eqref{eq:u-hat}. Next, we denote $\hat Z =
  (\hat{U}^T \otimes I) \hat Y (\hat U \otimes I)$ so that~\eqref{eq:y} reduces
  to
  \begin{equation*}
    \PAREN[\big]{I \otimes A^T - \widehat{\Lambda} \otimes B^T} \widehat{Z}
    + \widehat{Z} \PAREN[\big]{I \otimes A - \widehat{\Lambda} \otimes B}
    + \widehat{\Lambda}^2 \otimes I
    = 0,
  \end{equation*}
  which can be solved by taking
  \begin{equation*}
    \widehat{Z}
    =
    \DIAG(0, X_2, \ldots, X_k),
  \end{equation*}
  where again $X_i$, for $i = 2, \ldots, k$, is the observability Gramian of the
  auxiliary system $(A - \lambda_i B, E, \lambda_i I)$ in~\eqref{eq:aux-system}.
  Next,
  \[
    \hat X = \PAREN*{{\PTP}^\half \hat U \otimes I} \hat Z \PAREN*{\hat U^T
    {\PTP}^\half \otimes I}
  \]
  then satisfies~\eqref{eq:x1}, and it can be verified that ${\cX}_+(I \otimes A
  - \hat L \otimes B) \subset \ker \hat X$. Thus, the $\cH_2$-norm of $\hat S$
  is given by:
  \begin{equation}\label{eq:s-hat-norm-identity}
    \begin{aligned}
      \Htwonorm[\big]{\widehat{S}}^2
      &=
      \TR\Bigl(\PAREN*{\widehat{M}^T \PTP^\half \widehat{U} \otimes E^T}
      \DIAG(0, X_2, \ldots, X_k) \\
      &\qquad \quad
      \cdot \PAREN*{\widehat{U}^T \PTP^\half \widehat{M} \otimes E}\Bigr),
      \\
      &=
      \TR\Bigl(\PAREN*{\widehat{U}^T \PTP^\half \widehat{M} \widehat{M}^T
      \PTP^\half \widehat{U} \otimes I} \\
      &\qquad \quad
      \cdot \DIAG\PAREN*{0, E^T X_2 E, \ldots, E^T X_k E}\Bigr).
    \end{aligned}
  \end{equation}
  Using Lemma~\ref{lem:error-orthogonal}, and
  formulas~\eqref{eq:s-norm-identity} and~\eqref{eq:s-hat-norm-identity}, we
  compute
  \begin{equation}\label{eq:error-norm-trace}
    \begin{aligned}
      \Htwonorm[\big]{S - \widehat{S}}^2
      &= \TR\PAREN*{\PAREN*{U^T M M^T U \otimes I}
      \DIAG\PAREN*{0, E^T X_2 E, \ldots, E^T X_N E}} \\
      &\quad -
      \TR\Bigl(\PAREN*{\widehat{U}^T \PTP^\half \widehat{M}
      \widehat{M}^T \PTP^\half \widehat{U} \otimes I}\\
      &\qquad \qquad
      \cdot \DIAG\PAREN*{0, E^T X_2 E, \ldots, E^T X_k E}\Bigr) \\
      &= \TR\biggl(\PAREN*{
      \begin{pmatrix}
        U_1^T M M^T U_1 & U_1^T M M^T U_2 \\
        U_2^T M M^T U_1 & U_2^T M M^T U_2
      \end{pmatrix}
      \otimes I} \\
      &\qquad \quad
      \cdot \DIAG\PAREN*{0, E^T X_2 E, \ldots, E^T X_N E}\biggr) \\
      &\quad -
      \TR\PAREN*{\PAREN*{U_1^T M M^T U_1 \otimes I}
      \DIAG\PAREN*{0, E^T X_2 E, \ldots, E^T X_k E}} \\
      &=
      \TR\PAREN*{\PAREN*{U_2^T M M^T U_2 \otimes I}
      \DIAG\PAREN*{E^T X_{k + 1} E, \ldots, E^T X_N E}},
    \end{aligned}
  \end{equation}
  where the second equality follows from the fact that
  \begin{align*}
    \widehat{M}^T \PTP^\half \widehat{U}
    &=
    M^T P \PTP^{-1} \PTP^\half \widehat{U} \\
    &=
    M^T P \PTP^{-\half} \widehat{U} \\
    &=
    M^T U_1.
  \end{align*}
  Next, observe that~\eqref{eq:error-norm-trace} can be rewritten as
  \begin{equation*}
    \begin{aligned}
      \Htwonorm[\big]{S - \widehat{S}}^2
      &=
      \TR\PAREN*{\PAREN*{U_2^T M M^T U_2 \otimes I}
      \DIAG\PAREN*{E^T X_{k + 1} E, \ldots, E^T X_N E}} \\
      &=
      \TR\PAREN*{\PAREN*{U_2^T M M^T U_2}
      \DIAG\PAREN*{\TR\PAREN*{E^T X_{k+1} E}, \ldots, \TR\PAREN*{E^T X_N E}}} \\
      &=
      \TR\PAREN*{\PAREN*{U_2^T M M^T U_2}
      \DIAG\PAREN*{\Htwonorm{S_{\lambda_{k+1}}}^2, \ldots,
      \Htwonorm{S_{\lambda_N}}^2}},
    \end{aligned}
  \end{equation*}
  where $S_{\lambda_j}$ for $j = k + 1, \ldots, N$ is the transfer function of
  the auxiliary system~\eqref{eq:aux-system}. An upper bound for this expression
  is given by
  \begin{equation*}
    \begin{split}
      \TR\PAREN*{\PAREN*{U_2^T M M^T U_2}
      \DIAG\PAREN*{\Htwonorm{S_{\lambda_{k+1}}}^2, \ldots,
      \Htwonorm{S_{\lambda_N}}^2}}
      \leq
      S_{\max, \cH_2}^2 \TR\PAREN*{U_2^T M M^T U_2},
    \end{split}
  \end{equation*}
  where $S_{\max, \cH_2}^2 = \max_{k + 1 \le j \le N}
  \Htwonorm{S_{\lambda_j}}^2$. Furthermore, we have
  \begin{align*}
    \TR\PAREN*{U_2^T M M^T U_2}
    &=
    \TR\PAREN*{U^T M M^T U} - \TR\PAREN*{U_1^T M M^T U_1}
    \\
    &=
    m - \TR\PAREN*{P \PTP^{-1} P^T M M^T}.
  \end{align*}
  Since, by assumption, the partition $\pi$ is regularly formed, the matrix $P
  \PTP^{-1} P^T$ is a block diagonal matrix of the form
  \begin{equation*}
    P \PTP^{-1} P^T
    =
    \DIAG(P_1, P_2, \ldots, P_k).
  \end{equation*}
  It is easily verified that each $P_i$ is a $\card{C_i} \times \card{C_i}$
  matrix whose elements are all equal to $\frac{1}{\card{C_i}}$. The matrix $M
  M^T$ is a diagonal matrix whose diagonal entries are either $0$ or $1$. We
  then have that the $i$th column of $P \PTP^{-1} P^T M M^T$ is either equal to
  the $i$th column of $P \PTP^{-1} P^T$ if agent $i$ is a leader, or zero
  otherwise. It then follows that the diagonal elements of $P \PTP^{-1}P^T M
  M^T$ are either zero or $\frac{1}{\card{C_{k_i}}}$ if $i$ is part of the
  leader set, where $C_{k_i}$ is the cell containing agent $i$. Hence, we have
  \begin{equation*}
    \TR\PAREN*{U_1^T M M^T U_1}
    =
    \sum_{i = 1}^m \frac{1}{\card{C_{k_i}}},
  \end{equation*}
  and consequently
  \begin{equation*}
    \TR\PAREN*{U_2^T M M^T U_2}
    =
    m - \sum_{i = 1}^m \frac{1}{\card{C_{k_i}}}.
  \end{equation*}
  In conclusion, we have
  \begin{equation*}
    \Htwonorm[\big]{S - \widehat{S}}^2
    \leq
    S_{\max, \cH_2}^2 \sum_{i = 1}^m \PAREN*{1 - \frac{1}{\card{C_{k_i}}}},
  \end{equation*}
  which completes the proof of the first part of the theorem.

  We now prove the statement about the relative error. For this, we will
  establish a lower bound for $\Htwonorm{S}^2$. By~\eqref{eq:s-norm-identity} we
  have
  \begin{equation}\label{eq:s-norm-identity1}
    \begin{aligned}
      \Htwonorm{S}^2
      &=
      \TR\PAREN*{\PAREN*{M^T U \otimes E^T} \DIAG(0, X_2, \ldots, X_N)
      \PAREN*{U^T M \otimes E}} \\
      &=
      \TR\PAREN*{\PAREN*{U^T M M^T U \otimes I}
      \DIAG\PAREN*{0, E^T X_2 E, \ldots, E^T X_N E}} \\
      &=
      \TR\PAREN*{\PAREN*{U^T M M^T U}
      \DIAG\PAREN*{0, \TR\PAREN*{E^T X_2 E}, \ldots, \TR\PAREN*{E^T X_N E}}}.
    \end{aligned}
  \end{equation}
  The first column of $U$ spans the eigenspace corresponding to the eigenvalue
  $0$ of $L$ and hence must be equal to $u_1 = \frac{1}{\sqrt{N}} \bbone_N$. Let
  $\bar{U}$ be such that $U = \begin{pmatrix} u_1 & \bar{U} \end{pmatrix}$. It
  is then easily verified using~\eqref{eq:s-norm-identity1} that
  \begin{equation*}
    \begin{aligned}
      \Htwonorm{S}^2
      &=
      \TR\PAREN*{\PAREN*{\bar{U}^T M M^T \bar{U}}
      \DIAG\PAREN*{\TR\PAREN*{E^T X_2 E}, \ldots, \TR\PAREN*{E^T X_N E}}} \\
      &=
      \TR\PAREN*{\PAREN*{\bar{U}^T M M^T \bar{U}}
      \DIAG\PAREN*{\Htwonorm{S_{\lambda_{2}}}^2, \ldots,
      \Htwonorm{S_{\lambda_N}}^2}}.
    \end{aligned}
  \end{equation*}
  Finally, since
  \[
    \TR\PAREN*{\bar{U}^T M M^T \bar{U}}
    =
    \TR\PAREN*{M^T \bar{U} \bar{U}^T M}
    =
    \TR\PAREN*{M^T \PAREN*{U U^T - u_1 u_1^T} M}
    =
    m - \frac{m}{N},
  \]
  we obtain that $\Htwonorm{S}^2 \geq m \PAREN*{1 - \frac{1}{N}} S_{\min,
  \cH_2}^2$. This then yields the upper bound for the relative error as claimed.
\end{proof}

\begin{RM}
  Note that by our labeling of the eigenvalues of $L$, in the formulation of
  Theorem~\ref{thm:main}, we have that $\sigma(L) \setminus \sigma(\hat L)$ is
  equal to $\{\lambda_{k + 1}, \ldots, \lambda_N\}$ used in the proof. We stress
  that this should not be confused with the notation often used in the
  literature, where the $\lambda_i$'s are labeled in increasing order.
\end{RM}
\begin{RM}\label{rem:single integrator case}
  For the special case that the agents are single integrators (so $n = 1$, $A =
  0$, $B = 1$, and $E = 1$) it is easily seen that $S_{\max, \cH_2} =
  \frac{1}{2} \max \{ \lambda \mid \lambda \in \sigma(L) \setminus \sigma(\hat
  L) \}$ and $S_{\min, \cH_2} = \frac{1}{2} \min \{ \lambda \mid \lambda \in
  \sigma(L), ~\lambda \neq 0\}$. Thus, in the single integrator case the
  corresponding a priori upper bounds explicitly involve the Laplacian
  eigenvalues.

  As noted in the Introduction, the single integrator case was also studied
  in~\cite{morMonTC14} for the slightly different set up that the output
  equation in the original network~\eqref{eq:extended-system} is taken as $y =
  (W^\half R^T \otimes I) x$ instead of $y = (L \otimes I) x$. Here, $R$ is the
  incidence matrix of the graph and $W$ the diagonal matrix with the edge
  weights on the diagonal (in other words, $L = R W R^T$). It was shown in
  \cite{morMonTC14} that in that case the absolute and relative approximation
  errors admit the explicit expressions
  \begin{equation*}
    \Htwonorm{S - \widehat{S}}^2 =
    \frac{1}{2} \sum^m_{i = 1} \PAREN*{1 - \frac{1}{\card{C_{k_i}}}},
  \end{equation*}
  and
  \begin{equation*}
    \frac{\Htwonorm{S - \hat S}^2}{\Htwonorm{S}^2} =
    \frac{\sum^m_{i = 1} \PAREN*{1 - \frac{1}{\card{C_{k_i}}}}}{m \PAREN*{1 -
    \frac{1}{N}}}.
  \end{equation*}
\end{RM}

\section{\texorpdfstring{$\boldsymbol{\cH_\infty}$}{Hinf}-error bounds}
\label{sec:H_inf}

In the previous section, we obtained a priori upper bounds for the approximation
error in terms of the $\cH_2$-norm of the difference between the transfer
functions of the original network and its reduced order approximation. In the
present section, we express the error in terms of the $\cH_\infty$-norm.

\subsection{The single integrator case}

In this first subsection, we consider the special case that the agent dynamics
is a single integrator system. In this case, we have $A = 0$, $B = 1$, and $E =
1$ and the original system~\eqref{eq:extended-system} then reduces to
\begin{equation}\label{eq:extended-system-single}
  \begin{aligned}
    \dot{x} & = -L x + M u, \\
    y & = L x.
  \end{aligned}
\end{equation}
The state space dimension of~\eqref{eq:extended-system-single} is then simply
$N$, the number of agents. For a given partition $\pi = \{C_1, C_2, \ldots,
C_k\}$, the reduced system~\eqref{eq:reduced-system} is now given by
\begin{equation*}
  \begin{aligned}
    \dot{\widehat{x}} & = -\widehat{L} \widehat{x} + \widehat{M} u, \\
    \widehat{y} & = L P \widehat{x},
  \end{aligned}
\end{equation*}
where $P = P(\pi)$ is again the characteristic matrix of $\pi$ and $\hat x \in
\R^k$. The transfer functions $S$ and $\widehat{S}$, of the original and
reduced system respectively, are given by
\begin{align*}
  S(s)
  & =
  L {(s I_N + L)}^{-1} M, \\
  \widehat{S}(s)
  & =
  L P \PAREN[\big]{s I_k + \widehat{L}}^{-1} \widehat{M}.
\end{align*}
We then have the following explicit expressions for the $\cH_\infty$ model
reduction error:
\begin{TM}\label{TM:single_int_Hinf}
  Let $\pi$ be an almost equitable partition of the graph $\graph$. If the network with single
  integrator agent dynamics is clustered according to $\pi$, then the
  $\cH_\infty$-error is given by
  \begin{align*}
    \Hinfnorm[\big]{S - \hat{S}}^2
    =
    \begin{cases}
      \max\limits_{1 \le i \le m} \PAREN*{1 - \frac{1}{\card{C_{k_i}}}} &
      \text{if the leaders are in different cells,} \\
      1 & \text{otherwise},
    \end{cases}
  \end{align*}
  where $C_{k_i}$ is the set of cellmates of leader $i$ for some $k_i \in \{1,
  2, \ldots, k\}$. Furthermore, since $\Hinfnorm{S} = 1$, the relative and
  absolute $\cH_\infty$-errors coincide.
\end{TM}
\begin{RM}
  We see that the $\cH_\infty$-error lies in the interval $[0, 1]$. The error is
  maximal ($= 1$) if and only if two or more leader nodes occupy one and the
  same cell. The error is minimal ($= 0$) if and only if each leader node
  occupies a different cell, and is the only node in this cell. In general, the
  error decreases if the number of cellmates of the leaders decreases.
\end{RM}
\begin{proof}
  To simplify notation, denote $\Serr(s) = S(s) - \hat S(s)$. Note that both
  $S$ and $\hat S$ have all poles in the open left half plane. We now first show
  that since $\pi$ is an AEP we have
  \begin{equation}\label{eq:Hinfty-error-equality}
    \Hinfnorm{\Serr}
    =
    \twonorm{\Serr(0)}.
  \end{equation}
  First note that $\hat S(s) = L P \PTP^{-\frac{1}{2}} {(s I_k + \bar L)}^{-1}
  \PTP^{\frac{1}{2}} \hat M$, where the symmetric matrix $\bar L$ is given
  by~\eqref{eq:L-bar}. Thus, a state space representation for the error system
  is given by
  \begin{equation}\label{eq:error-system-single}
    \begin{aligned}
      \dot x_e
      &=
      \begin{pmatrix}
        -L & 0 \\
        0 & -\bar L
      \end{pmatrix}
      x_e
      +
      \begin{pmatrix}
        M \\
        \PTP^{\frac{1}{2}} \hat M
      \end{pmatrix} u \\
      e &=
      \begin{pmatrix}
        L & -L P \PTP^{-\frac{1}{2}}
      \end{pmatrix}
      x_e.
    \end{aligned}
  \end{equation}
  Next, we show that~\eqref{eq:Hinfty-error-equality} holds by applying
  Lemma~\ref{LM:suff_cond} to system~\eqref{eq:error-system-single}. Indeed,
  with $X = -L$ we have
  \begin{align*}
    &
    \begin{pmatrix}
      L & -L P \PTP^{-\frac{1}{2}}
    \end{pmatrix}
    \begin{pmatrix}
      -L & 0 \\
      0 & -\bar L
    \end{pmatrix} \\
    &\qquad =
    \begin{pmatrix}
      -L^2 & L P \PTP^{-\frac{1}{2}} \bar L
    \end{pmatrix} \\
    &\qquad =
    \begin{pmatrix}
      -L^2 & LP \hat L \PTP^{-\frac{1}{2}}
    \end{pmatrix} \\
    &\qquad =
    \begin{pmatrix}
      -L^2 & L^2 P \PTP^{-\frac{1}{2}}
    \end{pmatrix}
     = X
    \begin{pmatrix}
      L & -L P \PTP^{-\frac{1}{2}}
    \end{pmatrix},
  \end{align*}
  and from Lemma~\ref{LM:suff_cond} it then immediately follows that
  $\Hinfnorm{\Serr} = \twonorm{\Serr(0)}$. To compute $\twonorm{\Serr(0)}$ we
  apply Lemma~\ref{LM:H0} to system~\eqref{eq:error-system-single}. First, it is
  easily verified that
  \begin{equation*}
    \ker
      \begin{pmatrix}
        -L & 0 \\
        0 & -\bar L
      \end{pmatrix}
    \subset
    \ker
      \begin{pmatrix}
        L & -L P \PTP^{-\frac{1}{2}}
      \end{pmatrix}.
  \end{equation*}
  By applying Lemma~\ref{LM:H0} we then obtain
  \begin{equation}\label{eq:Serr0}
    \begin{aligned}
      \Serr(0)
      &=
      \begin{pmatrix}
        L & -L P \PTP^{-\frac{1}{2}}
      \end{pmatrix}
      \begin{pmatrix}
        L & 0 \\
        0 & \bar L
      \end{pmatrix}^+
      \begin{pmatrix}
        M \\ \PTP^{\frac{1}{2}} \hat M
      \end{pmatrix} \\
      &=
      L \PAREN*{L^+ - P \PTP^{-\frac{1}{2}} \bar{L}^+ \PTP^{-\frac{1}{2}}
      P^T} M.
    \end{aligned}
  \end{equation}
  Recall that $\hat U$ in~\eqref{eq:u-hat} is an orthogonal matrix that
  diagonalizes $\bar L$ and that $U_1 = P \PTP^{-\half} \hat U$. Then $\bar L^+
  = \hat U \hat \Lambda^+ \hat U^T$. Thus we have
  \begin{align*}
    P \PTP^{-\half} \bar{L}^+ \PTP^{-\half} P^T
    =
    U_1 \hat{\Lambda}^+ U_1^T.
  \end{align*}
  Next, we compute
  \begin{equation}\label{eq:LL-plus}
    \begin{aligned}
      L L^+
      &= U \Lambda U^T U \Lambda^+ U^T \\
      &= U \Lambda \Lambda^+ U^T \\
      &= I_N - \frac{1}{N} \bbone_N \bbone_N^T,
    \end{aligned}
  \end{equation}
  where the last equality follows from the fact that the first column of $U$ is
  $\frac{1}{\sqrt{N}} \bbone_N$. Next, observe that
  \begin{equation}\label{eq:LLambda-hat}
    \begin{aligned}
      L U_1 \hat \Lambda^+ U_1^T
      &= U \Lambda U^T U_1 \hat \Lambda^+ U_1^T \\
      &= U_1 \hat \Lambda \hat \Lambda^+ U_1^T \\
      &= U_1 U_1^T - \frac{1}{N} \bbone_N \bbone_N^T \\
      &= P \PTP^{-1} P^T - \frac{1}{N} \bbone_N \bbone_N^T.
    \end{aligned}
  \end{equation}
  Combining~\eqref{eq:LL-plus} and~\eqref{eq:LLambda-hat} with~\eqref{eq:Serr0},
  we obtain
  \begin{equation*}
    \Serr(0) = \PAREN*{I_N - P \PTP^{-1} P^T} M.
  \end{equation*}
  From~\eqref{eq:Hinfty-error-equality} then, we have that the
  $\cH_\infty$-error is given by
  \begin{equation}\label{eq:lam_max_min}
    \begin{aligned}
      \Hinfnorm[\big]{S - \hat{S}}^2
      & =
      \lambda_{\max} \PAREN*{{\Serr(0)}^T \Serr(0)}
      \\
      & =
      \lambda_{\max} \PAREN*{M^T \PAREN*{I_N - P \PTP^{-1} P^T}^2 M}
      \\
      & =
      \lambda_{\max} \PAREN*{I_m - M^T P \PTP^{-1} P^T M}
      \\
      & =
      1 - \lambda_{\min} \PAREN*{M^T P \PTP^{-1} P^T M}.
    \end{aligned}
  \end{equation}
  All that is left now is to compute the minimal eigenvalue of $M^T P \PTP^{-1}
  P^T M$. Again let $\{v_1, v_2, \ldots, v_m\}$ be the set of leaders and note
  that $M$ satisfies
  \begin{equation*}
    M
    =
    \begin{pmatrix}
      e_{v_1} & e_{v_2} & \cdots & e_{v_m}
    \end{pmatrix}.
  \end{equation*}
  Again, without loss of generality, assume that $\pi$ is regularly formed. Then
  the matrix $P \PTP^{-1} P^T$ is block diagonal where each diagonal block
  $P_i$
  is a $\card{C_i} \times \card{C_i}$ matrix whose entries are all
  $\frac{1}{\card{C_i}}$. Let $k_i \in \{1, 2, \ldots, k\}$ be such that $v_i
  \in C_{k_i}$. If all the leaders are in different cells, then
  \begin{equation*}
    M^T P \PTP^{-1} P^T M
    =
    \DIAG\PAREN*{\frac{1}{\card{C_{k_1}}}, \frac{1}{\card{C_{k_2}}}, \ldots,
    \frac{1}{\card{C_{k_m}}}},
  \end{equation*}
  and so
  \begin{equation}\label{eq:lam_min1}
    \lambda_{\min} \PAREN*{M^T P \PTP^{-1} P^T M}
    =
    \min_{1 \le i \le m} \frac{1}{\card{C_{k_i}}}.
  \end{equation}
  Now suppose that two leaders $v_i$ and $v_j$ are cellmates. Then we have
  \begin{align*}
    M^T P \PTP^{-1} P^T M (e_i - e_j)
    & =
    M^T P \PTP^{-1} P^T (e_{v_i} - e_{v_j})
    = 0.
  \end{align*}
  which together with $M^T P \PTP^{-1} P^T M \succeq 0$ implies
  \begin{align}
    \lambda_{\min}\PAREN*{M^T P \PTP^{-1} P^T M}
    =
    0. \label{eq:lam_min2}
  \end{align}
  From~\eqref{eq:lam_max_min},~\eqref{eq:lam_min1}, and~\eqref{eq:lam_min2}, we
  find the absolute $\cH_\infty$-error. To find the relative $\cH_\infty$-error,
  we compute $\Hinfnorm{S}$ by applying Lemma~\ref{LM:suff_cond} and
  Lemma~\ref{LM:H0} to the original system~\eqref{eq:extended-system-single}.
  Combined with~\eqref{eq:LL-plus}, this results in the $\cH_\infty$-norm of the
  original system:
  \begin{align*}
    \Hinfnorm{S}^2
    &= \lambda_{\max}\PAREN*{{S(0)}^T S(0)}
    =
    \lambda_{\max}\PAREN*{M^T \PAREN*{I_N - \frac{1}{N} \bbone_N \bbone_N^T} M}
    = 1.
  \end{align*}
\end{proof}

\subsection{The general case with symmetric agent dynamics}

In this subsection, we deal with the case that the agent dynamics is given by an
arbitrary multivariable system. The original and the reduced network are again
given by~\eqref{eq:extended-system} and~\eqref{eq:reduced-system}, respectively.
As in the proof of Theorem~\ref{TM:single_int_Hinf} we will rely heavily on
Lemma~\ref{LM:H0} to compute the $\cH_\infty$-error. Since Lemma~\ref{LM:H0}
relies on a symmetry argument, we will need to assume that the matrices $A$ and
$B$ are both symmetric, which will be a standing assumption in the remainder of
this section.

The main theorem of this section establishes an a priori upper bound for the
$\cH_\infty$-norm of the approximation error in the case that we cluster
according to an AEP. Again, an important role is played by the $N - 1$
auxiliary systems~\eqref{eq:aux-system} with $\lambda$ ranging over the nonzero
eigenvalues of the Laplacian $L$. Again, let $S_{\lambda}(s) = \lambda {(sI - A
+ \lambda B)}^{-1}E$ be their transfer functions We assume that the original
network~\eqref{eq:extended-system} is synchronized, so that all of the $A -
\lambda B$ are Hurwitz. We again use $S$, $\widehat{S}$, and $\Serr$ to denote
the relevant transfer functions.

We have the following theorem:
\begin{TM}
  Consider the network~\eqref{eq:extended-system} and assume that $A$ and $B$
  are symmetric matrices. Assume the network is synchronized. Let $\pi$ be an
  almost equitable partition of the graph $\graph$. The $\cH_\infty$-error when clustering
  $\graph$ according to $\pi$ then satisfies
  \begin{equation*}
    \Hinfnorm[\big]{S - \hat{S}}^2 \le
    \begin{cases}
      S_{\max, \cH_\infty}^2 \max\limits_{1 \le i \le m} \PAREN*{1 -
      \frac{1}{\card{C_{k_i}}}} &
      \text{if the leaders are in different cells,} \\
      S_{\max, \cH_\infty}^2 & \text{otherwise}
    \end{cases}
  \end{equation*}
  and
  \begin{equation*}
    \frac{\Hinfnorm[\big]{S - \hat{S}}^2}{\Hinfnorm{S}^2} \le
    \begin{cases}
      \frac{S_{\max, \cH_\infty}^2}{S_{\min, \cH_\infty}^2} \max\limits_{1 \le i
      \le m} \PAREN*{1 - \frac{1}{\card{C_{k_i}}}} &
      \text{if the leaders are in different cells,} \\
      \frac{S_{\max, \cH_\infty}^2}{S_{\min, \cH_\infty}^2} & \text{otherwise},
    \end{cases}
  \end{equation*}
  where
  \begin{equation}\label{eq:Smax,infinity}
    S_{\max, \cH_\infty}
    :=
    \max_{\lambda \in \sigma(L) \setminus \sigma(\widehat{L})}
    \Hinfnorm{S_\lambda}
  \end{equation}
  and
  \begin{equation}\label{eq:Smin,infinity}
    S_{\min, \cH_\infty}
    :=
    \min_{\lambda \in \sigma(L) \setminus \{0\}}
    \sigma_{\min}\PAREN*{S_\lambda(0)},
  \end{equation}
  with $S_\lambda$ the transfer function of the auxiliary
  system~\eqref{eq:aux-system}.
\end{TM}
\begin{RM}
  The absolute $\cH_\infty$-error thus lies in the interval $[0, S_{\max,
  \cH_\infty}]$ with $S_{\max, \cH_\infty}$ the maximum over the
  $\cH_\infty$-norms of the transfer functions $S_{\lambda}$ with $\lambda \in
  \sigma(L) \setminus \sigma(\widehat{L})$. The error is minimal ($= 0$) if each
  leader node occupies a different cell, and is the only node in this cell. In
  general, the upper bound decreases if the number of cellmates of the leaders
  decreases.
\end{RM}
\begin{proof}
  First note that the transfer function $\hat S$ of the reduced
  network~\eqref{eq:reduced-system} is equal to
  \begin{equation}\label{eq:alt}
    \begin{aligned}
      \hat S(s)
      &=
      \PAREN*{LP \PTP^{-\frac{1}{2}} \otimes I_n} \PAREN*{sI - I_k \otimes A +
      \bar L \otimes B}^{-1}
      \PAREN*{\PTP^{\frac{1}{2}} \hat M \otimes E},
    \end{aligned}
  \end{equation}
  with the symmetric matrix $\bar L$ given by~\eqref{eq:L-bar}. Analogous to the
  proof of Theorem~\ref{TM:single_int_Hinf}, we first apply
  Lemma~\ref{LM:suff_cond} to the error system
  \begin{equation*}
    \begin{aligned}
      \dot x_e
      &=
      \begin{pmatrix}
        I_N \otimes A - L \otimes B & 0 \\
        0 & I_k \otimes A - \bar L \otimes B
      \end{pmatrix}
      x_e
      +
      \begin{pmatrix}
        M \otimes E \\
        \PTP^{\frac{1}{2}} \hat M \otimes E
      \end{pmatrix}
      u \\
      e
      &=
      \begin{pmatrix}
        L \otimes I_n & -L P \PTP^{-\frac{1}{2}} \otimes I_n
      \end{pmatrix}
      x_e,
    \end{aligned}
  \end{equation*}
  with transfer function $\Serr$.
  Take $X = I_N \otimes A - L \otimes B$. We then have
  \begin{align*}
    &
    \begin{pmatrix}
      L \otimes I_n & -L P \PTP^{-\frac{1}{2}} \otimes I_n
    \end{pmatrix}
    \begin{pmatrix}
      I_N \otimes A - L \otimes B & 0 \\
      0 & I_k \otimes A -\bar L \otimes B
    \end{pmatrix} \\
    & \qquad = X
    \begin{pmatrix}
      L \otimes I_n & -L P \PTP^{-\frac{1}{2}} \otimes I_n
    \end{pmatrix}.
  \end{align*}
  From Lemma~\ref{LM:suff_cond}, we then obtain that
  $\Hinfnorm{\Serr}
  = \twonorm{\Serr(0)}
  = \lambda_{\max} \PAREN*{{\Serr(0)}^T \Serr(0)}^{\frac{1}{2}}$.

  In the proof of Lemma~\ref{lem:error-orthogonal}, it was shown that
  \begin{align*}
    {\hat S}{(-s)}^T \Serr(s) = {\hat S}{(-s)}^T (S(s) - \hat S(s)) = 0.
  \end{align*}
  Since all transfer
  functions involved are stable, in particular this holds for $s = 0$. We then
  have that $\widehat{S}{(0)}^T (S(0) - \widehat{S}(0)) = 0$, i.e.\
  $\widehat{S}{(0)}^T S(0) = \widehat{S}{(0)}^T \widehat{S}(0)$. By transposing,
  we also have ${S(0)}^T \widehat{S}(0) = \widehat{S}{(0)}^T \widehat{S}(0)$.
  Therefore,
  \begin{align*}
    {\Serr(0)}^T \Serr(0)
    & =
    \PAREN[\big]{S(0) - \widehat{S}(0)}^T (S(0) - \widehat{S}(0)) \\
    & =
    {S(0)}^T S(0) - {S(0)}^T \widehat{S}(0) - \widehat{S}{(0)}^T S(0)
    + \widehat{S}{(0)}^T \widehat{S}(0) \\
    & =
    {S(0)}^T S(0) - \widehat{S}{(0)}^T \widehat{S}(0).
  \end{align*}
  By applying Lemma~\ref{LM:H0} to system~\eqref{eq:extended-system}, we obtain
  \begin{equation}\label{eq:s0s0}
    \begin{aligned}
      {S(0)}^T S(0)
      & =
      \PAREN*{M^T \otimes E^T} {(I_N \otimes A - L \otimes B)}^+
      \PAREN*{L^2 \otimes I_n} \\
      & \qquad \cdot
      {(I_N \otimes A - L \otimes B)}^+ (M \otimes E) \\
      & =
      \PAREN*{M^T \otimes E^T} (U \otimes I_n) {(I_N \otimes A - \Lambda
      \otimes B)}^+ \PAREN*{\Lambda^2 \otimes I_n} \\
      & \qquad \cdot
      {(I_N \otimes A - \Lambda \otimes B)}^+ \PAREN*{U^T \otimes I_n} (M
      \otimes E) \\
      & =
      \PAREN*{M^T U \otimes E^T}
      \DIAG\PAREN*{0, \lambda_2^2 {(A - \lambda_2 B)}^{-2}, \ldots, \lambda_N^2
      {(A - \lambda_N B)}^{-2}} \\
      & \qquad \cdot
      (U^T M \otimes E) \\
      & =
      \PAREN*{M^T U \otimes I_r}
      \DIAG\PAREN*{0, {S_{\lambda_2}(0)}^T S_{\lambda_2}(0), \ldots,
      {S_{\lambda_N}(0)}^T S_{\lambda_N}(0)} \\
      & \qquad \cdot
      \PAREN*{U^T M \otimes I_r},
    \end{aligned}
  \end{equation}
  where $S_{\lambda}$ is again the transfer function of the auxiliary
  system~\eqref{eq:aux-system}. Recall that $\hat M = \PTP^{-1} P^T M$ and $U_1
  = P \PTP^{-\half} \hat U$. We now apply Lemma~\ref{LM:H0} to the transfer
  function~\eqref{eq:alt} of the system~\eqref{eq:reduced-system}:
  \begin{align*}
    \widehat{S}{(0)}^T \widehat{S}(0)
    & =
    \PAREN*{M^T P \PTP^{-\half} \otimes E^T} \PAREN[\big]{I_N \otimes A -
    \bar{L} \otimes B}^+ \\
    & \qquad \cdot
    \PAREN*{\PTP^{-\half} P^T L^2 P \PTP^{-\half} \otimes I_n} \\
    & \qquad \cdot
    \PAREN*{I_N \otimes A - \bar{L} \otimes B}^+ \PAREN*{\PTP^{-\half} P^T
    M \otimes E} \\
    & =
    \PAREN*{M^T P \PTP^{-\half} \otimes E^T} \PAREN[\big]{\hat{U} \otimes I_n}
    \PAREN[\big]{I_N \otimes A - \hat{\Lambda} \otimes B}^+ \\
    & \qquad \cdot
    \PAREN*{\hat{\Lambda}^2 \otimes I_n}
    \PAREN*{I_N \otimes A - \hat{\Lambda} \otimes B}^+ \\
    & \qquad \cdot
    \PAREN[\big]{\hat{U}^T \otimes I_n} \PAREN*{\PTP^{-\half} P^T M \otimes E}
    \\
    & =
    \PAREN*{M^T U_1 \otimes E^T} \\
    & \qquad \cdot
    \DIAG\PAREN*{0, \lambda_2^2 {(A - \lambda_2 B)}^{-2}, \ldots, \lambda_k^2
    {(A - \lambda_k B)}^{-2}} \\
    & \qquad \cdot
    \PAREN*{U_1^T M \otimes E} \\
    & =
    \PAREN*{M^T U_1 \otimes I_r} \\
    & \qquad \cdot
    \DIAG\PAREN*{0, {S_{\lambda_2}(0)}^T S_{\lambda_2}(0), \ldots,
    {S_{\lambda_k}(0)}^T S_{\lambda_k}(0)} \\
    & \qquad \cdot
    \PAREN*{U_1^T M \otimes I_r}.
  \end{align*}
  Combining the two expression above, it immediately follows that
  \begin{align*}
    {\Serr(0)}^T \Serr(0)
    &=
    {S(0)}^T S(0) - {\hat S}{(0)}^T \hat S(0) \\
    & =
    \PAREN*{M^T U_2 \otimes I_r} \\
    & \qquad \cdot
    \DIAG\PAREN*{{S_{\lambda_{k+1}}(0)}^T S_{\lambda_{k+1}}(0), \ldots,
    {S_{\lambda_N}(0)}^T S_{\lambda_N}(0)} \\
    & \qquad \cdot
    \PAREN*{U_2^T M \otimes I_r}.
  \end{align*}
  By taking $S_{\max, \cH_\infty}$ as defined by~\eqref{eq:Smax,infinity} it
  then holds that
  \begin{align*}
    {\Serr(0)}^T \Serr(0)
    & \preceq
    \PAREN*{M^T U_2 \otimes I_r}
    \DIAG(S_{\max, \cH_\infty}^2 I_r, \ldots, S_{\max, \cH_\infty}^2 I_r)
    \PAREN*{U_2^T M \otimes I_r} \\
    & =
    S_{\max, \cH_\infty}^2 \PAREN*{M^T U_2 U_2^T M \otimes I_r} \\
    & =
    S_{\max, \cH_\infty}^2 \PAREN*{M^T (I_N - U_1 U_1^T) M \otimes I_r} \\
    & =
    S_{\max, \cH_\infty}^2 \PAREN*{\PAREN*{I_m - M^T P \PTP^{-1} P^T
    M} \otimes I_r}.
  \end{align*}
  Continuing as in the proof of Theorem~\ref{TM:single_int_Hinf}, we find an
  upper bound for the $\cH_\infty$-error:
  \begin{equation*}
    \Hinfnorm{\Delta}^2 \leq S_{\max, \cH_\infty}^2 \lambda_{\max} \PAREN*{I_m -
    M^T P \PTP^{-1} P^T M}.
  \end{equation*}
  To compute an upper bound for the relative $\cH_\infty$-error, we bound the
  $\cH_\infty$-norm of system~\eqref{eq:extended-system} from below. Again, let
  $\bar U$ be such that $U = \begin{pmatrix} u_1 & \bar U \end{pmatrix}$ and let
  $S_{\min, \cH_\infty}$ be as defined by~\eqref{eq:Smin,infinity}.
  From~\eqref{eq:s0s0} it now follows that
  \begin{align*}
    {S(0)}^T S(0)
    & =
    \PAREN*{M^T \bar U \otimes I_r}
    \DIAG\PAREN*{{S_{\lambda_2}(0)}^T S_{\lambda_2}(0), \ldots,
    {S_{\lambda_N}(0)}^T S_{\lambda_N}(0)}
    \PAREN*{\bar U^T M \otimes I_r} \\
    & \geq
    \PAREN*{M^T \bar U \otimes I_r}
    \DIAG\PAREN*{S_{\min, \cH_\infty}^2 I_r, \ldots, S_{\min, \cH_\infty}^2 I_r}
    \PAREN*{\bar U^T M \otimes I_r} \\
    & =
    S_{\min, \cH_\infty}^2
    \PAREN*{M^T \bar U \bar U M \otimes I_r} \\
    & =
    S_{\min, \cH_\infty}^2
    \PAREN*{M^T \PAREN*{I_N - \frac{1}{N} \bbone_N \bbone_N^T} M \otimes I_r}.
  \end{align*}
  Again using Lemma~\ref{LM:H0}, we find a lower bound to the $\cH_\infty$-norm
  of $S$:
  \begin{equation*}
    \Hinfnorm{S}^2
    =
    \lambda_{\max} \PAREN*{{S(0)}^T S(0)} \geq S_{\min, \cH_\infty}^2,
  \end{equation*}
  which concludes the proof of our theorem.
\end{proof}

\section{Towards a priori error bounds for general graph
partitions}\label{sec:general}

Up to now, in this paper we have dealt with establishing a priori error bounds
for network reduction by clustering using AEPs of the network graph. Of
course, an important problem is to find error bounds for \emph{arbitrary},
possibly non almost equitable, partitions. In this section, we address this more
general problem. We restrict ourselves to the special case that the agents have
single integrator dynamics. Thus, we consider the multi-agent network
\begin{equation}\label{eq:original}
  \begin{aligned}
    \dot{x} & = -L x + M u, \\
    y & = L x.
  \end{aligned}
\end{equation}
As before, we assume that the underlying (undirected) graph $\graph$ is
connected, so that the network is synchronized. Assume $\pi = \{C_1, C_2,
\ldots, C_k\}$ is a graph partition, not necessarily an AEP, and let $P =
P(\pi) \in \mathbb{R}^{N \times k}$ be its characteristic matrix. As before, the
reduced order network is taken to be the Petrov-Galerkin projection
of~\eqref{eq:original}, and is represented by
\begin{equation}\label{eq:reduced}
  \begin{aligned}
    \dot{\widehat{x}} & = -\widehat{L} \widehat{x} + \widehat{M} u, \\
    \widehat{y} & = L P \widehat{x},
  \end{aligned}
\end{equation}
Again, let $S$ and $\hat S$ be the transfer functions of~\eqref{eq:original}
and~\eqref{eq:reduced}, respectively. We address the problem of obtaining a
priori upper bounds for $\Htwonorm[\big]{S - \widehat{S}}$ and
$\Hinfnorm[\big]{S - \widehat{S}}$.

The idea for establishing such upper bounds is as follows: as a first step we
will approximate the original Laplacian matrix $L$ (of the original network
graph $\graph$) by a new Laplacian matrix, denoted by $L_\aeptxt$ (corresponding
to a `nearby' graph $\graph_\aeptxt$) such that the given partition $\pi$ is
an AEP with respect to this new graph $\graph_\aeptxt$. This new graph
$\graph_\aeptxt$ defines a new multi-agent system with transfer function
$S_\aeptxt(s) = L_\aeptxt {(s I + L_\aeptxt)}^{-1} M$. The reduced order network
of $S_\aeptxt$ (using the AEP $\pi$) has transfer function
$\widehat{S}_\aeptxt(s) = L_\aeptxt P \PAREN[\big]{s I +
\widehat{L}_\aeptxt}^{-1} \widehat{M}$. Then using the triangle inequality both
for $p = 2$ and $p = \infty$ we have
\begin{equation}\label{eq:triangle}
  \begin{aligned}
    \Hpnorm[\big]{S - \widehat{S}}
    & =
    \Hpnorm[\big]{S - S_\aeptxt + S_\aeptxt - \widehat{S}_\aeptxt +
    \widehat{S}_\aeptxt - \widehat{S}}
    \\
    & \le
    \Hpnorm{S - S_\aeptxt} + \Hpnorm[\big]{S_\aeptxt - \widehat{S}_\aeptxt} +
    \Hpnorm[\big]{\widehat{S}_\aeptxt - \widehat{S}}.
  \end{aligned}
\end{equation}
The idea is to obtain a priori upper bounds for all three terms
in~\eqref{eq:triangle}. We first propose an approximating Laplacian matrix
$L_\aeptxt$, and subsequently study the problems of establishing upper bounds
for the three terms in~\eqref{eq:triangle} separately.

For a given matrix $M$, let $\Fnorm{M} := \TR\PAREN*{M^T M}^\half$ denote its
Frobenius norm. In the following, denote $\cP := P \PTP^{-1} P^T$. Note that
$\cP$ is the orthogonal projector onto $\IM P$. As approximation for $L$, we
compute the unique solution to the convex optimization problem
\begin{equation}\label{eq:optimization-problem}
  \begin{aligned}
    \MIN_{L_\aeptxt} \quad & \Fnorm{L - L_\aeptxt}^2, \\
    \textnormal{subject to} \quad & (I_N - \cP) L_\aeptxt P = 0, \\
    & L_\aeptxt = L_\aeptxt^T, \\
    & L_\aeptxt \geq 0, \\
    & L_\aeptxt \bbone_N = 0.
  \end{aligned}
\end{equation}
In other words, we want to compute a positive semi-definite matrix $L_\aeptxt$
with row sums equal to zero, and with the property that $\IM P$ is invariant
under $L_\aeptxt$ (equivalently, the given partition $\pi$ is an AEP for the
new graph). We will show that such $L_\aeptxt$ may correspond to an undirected
graph \emph{with negative weights}. However, it is constrained to be positive
semi-definite, so the results of Sections~\ref{sec:partitions}, \ref{sec:H_2},
and~\ref{sec:H_inf} in this paper will remain valid.
\begin{TM}
  The matrix $L_\aeptxt := \cP L \cP + (I_N - \cP) L (I_N - \cP)$ is the unique
  solution to the convex optimization
  problem~\eqref{eq:optimization-problem}. If $L$ corresponds to a connected
  graph, then, in fact, $\KER{L_\aeptxt} = \IM \bbone_N$.
\end{TM}
\begin{proof}
  Clearly, $L_\aeptxt$ is symmetric and positive semi-definite since $L$ is.
  Also, $(I_N - \cP) L_\aeptxt P = 0$ since $(I_N - \cP) P = 0$. It is also
  obvious that $L_\aeptxt \bbone_N = 0$ since $\cP \bbone_N = \bbone_N$. We now
  show that $L_\aeptxt$ uniquely minimizes the distance to $L$. Let $X$ satisfy
  the constraints and define $\Delta = L_\aeptxt - X$. Then we have
  \[
    \Fnorm{L - X}^2
    = \Fnorm{L - L_\aeptxt}^2 + \Fnorm{\Delta}^2
    + 2 \TR\PAREN*{(L - L_\aeptxt) \Delta}.
  \]
  It can be verified that $L - L_\aeptxt = (I_N - \cP) L \cP + \cP L (I_N -
  \cP)$. Thus,
  \begin{align*}
    \TR\PAREN*{(L - L_\aeptxt) \Delta}
    = \TR\PAREN*{(I_N - \cP) L \cP \Delta}
    + \TR\PAREN*{\cP L (I_N - \cP) \Delta}.
  \end{align*}
  Now, since both $X$ and $L_\aeptxt$ satisfy the
  first constraint, we have $(I_N - \cP) \Delta \cP = 0$. Using this we have
  \begin{align*}
    \TR\PAREN*{(I_N - \cP) L \cP \Delta}
    = \TR\PAREN*{\cP \Delta (I_N - \cP) L}
    = \TR\PAREN*{L (I_N - \cP) \Delta \cP}
    = 0.
  \end{align*}
  Also,
  \begin{align*}
    \TR\PAREN*{\cP L (I_N - \cP) \Delta}
    = \TR\PAREN*{L (I_N - \cP) \Delta \cP}
    = 0.
  \end{align*}
  Thus, we obtain
  \begin{equation*}
    \Fnorm{L - X}^2
    = \Fnorm{L - L_\aeptxt}^2 + \Fnorm{\Delta}^2,
  \end{equation*}
  from which it follows that $\Fnorm{L - X}$ is minimal if and only if $\Delta =
  0$, equivalently $X = L_\aeptxt$.

  To prove the second statement, let $x \in \KER L_\aeptxt$, so $x^T L_\aeptxt x
  = 0$. Then both $x^T \cP L \cP x = 0$ and $x^T (I_N - \cP) L (I_N - \cP) x =
  0$. This clearly implies $L \cP x = 0$ and $L (I_N - \cP) x = 0$. Since $L$
  corresponds to a connected graph we must have $\cP x \in \IM \bbone_N$ and
  $(I_N - \cP) x \in \IM \bbone_N$. We conclude that $x \in \IM \bbone_N$ as
  desired.
\end{proof}
As announced above, $L_\aeptxt$ may have positive off-diagonal elements,
corresponding to a graph with some of its edge weights being negative. For
example, for
\begin{align*}
  L & =
  \begin{pmatrix}
    1 & -1 & 0 & 0 & 0 \\
    -1 & 2 & -1 & 0 & 0 \\
    0 & -1 & 2 & -1 & 0 \\
    0 & 0 & -1 & 2 & -1 \\
    0 & 0 & 0 & -1 & 1
  \end{pmatrix}, \quad
  P =
  \begin{pmatrix}
    1 & 0 \\
    1 & 0 \\
    1 & 0 \\
    0 & 1 \\
    0 & 1
  \end{pmatrix}
\end{align*}
we have
\begin{align*}
  L_\aeptxt & =
  \begin{pmatrix}
    \frac{11}{9} & -\frac{7}{9} & -\frac{1}{9} & 0 & -\frac{1}{3} \\
    -\frac{7}{9} & \frac{20}{9} & -\frac{10}{9} & 0 & -\frac{1}{3} \\
    -\frac{1}{9} & -\frac{10}{9} & \frac{14}{9} & -\half & \frac{1}{6} \\
    0 & 0 & -\half & \frac{3}{2} & -1 \\
    -\frac{1}{3} & -\frac{1}{3} & \frac{1}{6} & -1 & \frac{3}{2}
  \end{pmatrix},
\end{align*}
so the edge between nodes $3$ and $5$ has a negative weight.
Figure~\ref{fig:L_AEP_example} shows the graphs corresponding to $L$ and
$L_\aeptxt$.
\begin{figure}[tb]
  \centering
  \begin{tikzpicture}
    \foreach \name/\pos in {{1/(0, 0)}, {2/(2, 0)}, {3/(4, 0)}, {4/(3, -2)},
    {5/(1, -2)}}
      \node[shape=circle, ball color=black!5, inner sep=2pt, outer sep=0pt,
      minimum size=16pt] (\name) at \pos {$\name$};

    \foreach \src/\dest in {{1/2}, {2/3}, {3/4}, {4/5}}
      \draw (\src) -- node[fill=white, inner sep=2pt, font=\footnotesize] {$1$}
      (\dest);

    \begin{scope}[shift={(6.5, 0)}]
      \foreach \name/\pos in {{1/(0, 0)}, {2/(2, 0)}, {3/(4, 0)}, {4/(3, -2)},
      {5/(1, -2)}}
        \node[shape=circle, ball color=black!5, inner sep=2pt, outer sep=0pt,
        minimum size=16pt] (\name) at \pos {$\name$};

      \foreach \src/\dest/\weight in {{1/2/$\frac{7}{9}$}, {1/5/$\frac{1}{3}$},
      {2/3/$\frac{10}{9}$}, {2/5/$\frac{1}{3}$}, {3/4/$\half$},
      {3/5/$-\frac{1}{6}$}, {4/5/$1$}}
        \draw (\src) -- node[fill=white, inner sep=2pt, font=\footnotesize]
        {\weight} (\dest);
      \draw (1) edge[bend left=45] node[fill=white, inner sep=2pt,
      font=\footnotesize] {$\frac{1}{9}$} (3);
    \end{scope}
  \end{tikzpicture}
  \caption{Example}\label{fig:L_AEP_example}
\end{figure}
Although $L_\aeptxt$ is not necessarily a Laplacian matrix with only nonpositive
off-diagonal elements, it has all the properties we associate with a Laplacian
matrix. Specifically, it can be checked that all results in this paper remain
valid, since they only depend on the symmetric positive semi-definiteness of the
Laplacian matrix.

Using the approximating Laplacian $L_\aeptxt = \cP L \cP + (I_N - \cP) L (I_N -
\cP)$ as above, we will now deal with establishing upper bounds for the three
terms in~\eqref{eq:triangle}. We start off with the middle term
$\Hpnorm[\big]{S_\aeptxt - \widehat{S}_\aeptxt}$ in~\eqref{eq:triangle}.

According to Remark~\ref{rem:single integrator case}, for $p = 2$ this term has
an upper bound depending on the maximal eigenvalue of $L_\aeptxt$ that is not an
eigenvalue of $\widehat{L}_\aeptxt$, on the minimal nonzero eigenvalue of
$L_\aeptxt$, and on the number of cellmates of the leaders with respect to the
partitioning $\pi$.

For $p = \infty$, in Theorem~\ref{TM:single_int_Hinf} this term was expressed in
terms of the maximal number of cellmates with respect to the partitioning $\pi$
(noting that it is equal to $1$ in case two or more leaders share the same
cell).

Next, we will take a look at the first and third term in~\eqref{eq:triangle},
i.e.\ $\Hpnorm{S - S_\aeptxt}$ and $\Hpnorm[\big]{\widehat{S} -
\widehat{S}_\aeptxt}$. Let us denote $\Delta L = L - L_\aeptxt$. We find
\begin{align*}
  S(s) - S_\aeptxt(s)
  & =
  L {(s I + L)}^{-1} M - L_\aeptxt {(s I + L_\aeptxt)}^{-1} M
  \\
  & =
  L {(s I + L)}^{-1} M
  \\
  & \quad -
  L_\aeptxt \BRACK*{{(s I + L)}^{-1} + {(s I + L_\aeptxt)}^{-1} \Delta L {(s I +
  L)}^{-1}} M \\
  & =
  L {(s I + L)}^{-1} M - L_\aeptxt {(s I + L)}^{-1} M \\
  & \quad -
  L_\aeptxt {(s I + L_\aeptxt)}^{-1} \Delta L {(s I + L)}^{-1} M \\
  & =
  \Delta L {(s I + L)}^{-1} M - L_\aeptxt {(s I + L_\aeptxt)}^{-1} \Delta L {(s
  I + L)}^{-1} M \\
  & =
  \BRACK*{I_N - L_\aeptxt {(s I + L_\aeptxt)}^{-1}} \Delta L {(s I + L)}^{-1} M.
\end{align*}
Thus, for $p = 2$ and $p = \infty$ we have
\begin{equation}\label{eq:first}
  \begin{aligned}
    \Hpnorm{S - S_\aeptxt}
    & \le
    \Hinfnorm*{I_N - L_\aeptxt {(s I + L_\aeptxt)}^{-1}}
    \Hpnorm*{\Delta L {(s I + L)}^{-1} M} \\
    & \le
    2 \Hpnorm*{\Delta L {(s I + L)}^{-1} M}.
  \end{aligned}
\end{equation}
It is also easily seen that $\widehat{L}_\aeptxt = \PTP^{-1} P^T L_\aeptxt P =
\PTP^{-1} P^T L P = \widehat{L}$ and $L_\aeptxt P = P \PTP^{-1} P^T L P = P
\widehat{L}$. Therefore,
\begin{align*}
  \widehat{S}(s) - \widehat{S}_\aeptxt(s)
  & =
  L P \PAREN[\big]{s I + \widehat{L}}^{-1} \widehat{M} - L_\aeptxt P
  \PAREN[\big]{s I + \widehat{L}_\aeptxt}^{-1} \widehat{M} \\
  & =
  L P \PAREN[\big]{s I + \widehat{L}}^{-1} \widehat{M} - P \widehat{L}
  \PAREN[\big]{s I + \widehat{L}}^{-1} \widehat{M} \\
  & =
  \PAREN[\big]{L P - P \widehat{L}} \PAREN[\big]{s I + \widehat{L}}^{-1}
  \widehat{M}.
\end{align*}
Since, finally, ${(LP -P\widehat{L})}^T (LP -P\widehat{L}) = P^T {(\Delta L)}^2
P$, for $p = 2$ and $p = \infty$ we obtain
\begin{equation} \label{eq:third}
  \begin{aligned}
    \Hpnorm[\big]{\widehat{S} - \widehat{S}_\aeptxt}
    \le
    \Hpnorm*{\Delta L P{(s I + \widehat{L})}^{-1} \widehat{M}}.
  \end{aligned}
\end{equation}
Thus, both in~\eqref{eq:first} and~\eqref{eq:third} the upper bound involves the
difference $\Delta L = L - L_\aeptxt$ between the original Laplacian and its
optimal approximation in the set of Laplacian matrices for which the given
partition $\pi$ is an AEP. In a sense, the difference $\Delta L$ measures
how far $\pi$ is away from being an AEP for the original graph $\graph$.
Obviously, $\Delta L = 0$ if and only if $\pi$ is an AEP for $\graph$. In
that case only the middle term in~\eqref{eq:triangle} is present.

\section{Conclusions}\label{sec:conclusions}

In this paper, we have extended results on model reduction of leader-follower
networks with single integrator agent dynamics to leader-follower networks with
arbitrary linear multivariable agent dynamics. The proposed model reduction
technique reduces the complexity of the network topology by clustering the
agents according to a special class of graph partitions called almost equitable
partitions. We have shown that if the original undirected network is reduced by
means of a specific Petrov-Galerkin projection associated with such graph
partition, then the resulting reduced order model can be interpreted as a
networked multi-agent system with a weighted, directed network graph. If the
original network is clustered according to an almost equitable partition, then
its consensus properties are preserved. We have provided a priori upper bounds
on the $\cH_2$ and $\cH_\infty$ model reduction errors in this case. These error
bounds depend on an auxiliary system closely related to the agent dynamics, the
eigenvalues of the Laplacian matrices of the original and the reduced network,
and on the number of cellmates of the leaders in the network. Finally, we have
provided some insight into the general case of clustering according to
arbitrary, not necessarily almost equitable, partitions. Here, direct
computation of a priori upper bounds on the error is not as straightforward as
in the case of almost equitable partitions. We have shown that in this more
general case one can bound the model reduction errors by first optimally
approximating the original network by a new network for which the chosen
partition is almost equitable, and then bounding the $\cH_2$ and $\cH_\infty$
errors using the triangle inequality.

\bibliographystyle{abbrv}
\bibliography{root}

\begin{thebibliography}{10}

\bibitem{morAnt05}
A.~Antoulas.
\newblock {\em Approximation of Large-Scale Dynamical Systems}, volume~6 of
  {\em Advances in Design and Control}.
\newblock {SIAM} Publications, Philadelphia, PA, 2005.

\bibitem{BenFFetal14}
P.~Benner, R.~Findeisen, D.~Flockerzi, U.~Reichl, and K.~Sundmacher, editors.
\newblock {\em Large-Scale Networks in Engineering and Life Sciences}.
\newblock Modeling and Simulation in Science, Engineering and Technology.
  Birkh{\"a}user, Basel, CH, 2014.

\bibitem{morBenKS13}
P.~Benner, P.~K{\"u}rschner, and J.~Saak.
\newblock An improved numerical method for balanced truncation for symmetric
  second order systems.
\newblock {\em Math. Comput. Model. Dyn. Syst.}, 19(6):593--615, 2013.

\bibitem{morBenMS05}
P.~Benner, V.~Mehrmann, and D.~C. Sorensen.
\newblock {\em Dimension Reduction of Large-Scale Systems}, volume~45 of {\em
  Lect. Notes Comput. Sci. Eng.}
\newblock Springer-Verlag, Berlin/Heidelberg, Germany, 2005.

\bibitem{morBesSJ14}
B.~Besselink, H.~Sandberg, and K.~H. Johansson.
\newblock Model reduction of networked passive systems through clustering.
\newblock In {\em Proc. European Control Conf. ECC~2014, Strasbourg}, pages
  1069--1074, June 2014.

\bibitem{CarDR07}
D.~M. Cardoso, C.~Delorme, and P.~Rama.
\newblock Laplacian eigenvectors and eigenvalues and almost equitable
  partitions.
\newblock {\em European J. Combin.}, 28(3):665--673, 2007.

\bibitem{CorMKetal02}
J.~Cortes, S.~Martinez, T.~Karatas, and F.~Bullo.
\newblock Coverage control for mobile sensing networks.
\newblock In {\em IEEE International Conference on Robotics and Automation,
  2002. Proceedings. ICRA'02.}, volume~2, pages 1327--1332. IEEE, 2002.

\bibitem{EgeMCetal12}
M.~Egerstedt, S.~Martini, M.~Cao, M.~K. \c{C}amlibel, and A.~Bicchi.
\newblock Interacting with networks: {H}ow does structure relate to
  controllability in single-leader, consensus networks?
\newblock {\em {IEEE} Control Syst. Mag.}, 32(4):66--73, 2012.

\bibitem{EstFHetal10}
E.~Estrada, M.~Fox, D.~J. Higham, and G.-L. Oppo, editors.
\newblock {\em Network Science: Complexity in Nature and Technology}.
\newblock Springer-Verlag, London, UK, 2010.

\bibitem{FaxM04}
J.~A. Fax and R.~M. Murray.
\newblock Information flow and cooperative control of vehicle formations.
\newblock {\em {IEEE} Trans. Automat. Control}, 49(9):1465--1476, 2004.

\bibitem{GaoLDetal14}
J.~Gao, Y.-Y. Liu, R.~M. D'Souza, and A.-L. Barab{\'a}si.
\newblock Target control of complex networks.
\newblock {\em Nat. Commun.}, 5, 2014.

\bibitem{morImu12}
J.-I. Imura.
\newblock Clustered model reduction of large-scale complex networks.
\newblock In {\em Proceeding of the 20th International Symposium on
  Mathematical Theory of Networked and Systems, Melbourne, Australia}, 2012.

\bibitem{morIshKIetal12a}
T.~Ishizaki, K.~Kashima, J.-I. Imura, and K.~Aihara.
\newblock Model reduction of multi-input dynamical networks based on
  clusterwise controllability.
\newblock In {\em Proc. Am. Control Conf.}, pages 2301--2306. IEEE, June 2012.

\bibitem{morIshKIetal14}
T.~Ishizaki, K.~Kashima, J.-I. Imura, and K.~Aihara.
\newblock Model reduction and clusterization of large-scale bidirectional
  networks.
\newblock {\em {IEEE} Trans. Autom. Control}, 59(1):48--63, Jan. 2014.

\bibitem{morJonTC15}
H.~J. Jongsma, H.~L. Trentelman, and M.~K. \c{C}amlibel.
\newblock Model reduction of consensus networks by graph simplification.
\newblock In {\em 54th IEEE Conference on Decision and Control, Osaka, Japan},
  pages 5340--5345. IEEE, 2015.

\bibitem{morLalKM03}
S.~Lall, P.~Krysl, and J.~E. Marsden.
\newblock Structure-preserving model reduction for mechanical systems.
\newblock {\em Phys. D}, 184(1-4):304--318, 2003.

\bibitem{morLiB04}
R.-C. Li and Z.~Bai.
\newblock Structure-preserving model reduction.
\newblock In J.~Dongarra, K.~Madsen, and J.~Wa{\'s}niewski, editors, {\em
  Applied Parallel Computing. State of the Art in Scientific Computing: 7th
  International Workshop, PARA 2004, Lyngby, Denmark, June 20-23, 2004. Revised
  Selected Papers}, pages 323--332. Springer Berlin Heidelberg, 2006.

\bibitem{LiDCetal10}
Z.~Li, Z.~Duan, G.~Chen, and L.~Huang.
\newblock Consensus of multiagent systems and synchronization of complex
  networks: A unified viewpoint.
\newblock {\em IEEE Trans. Circuits Syst. I, Regular Papers}, 57(1):213--224,
  Jan. 2010.

\bibitem{MaZ10}
C.-Q. Ma and J.-F. Zhang.
\newblock Necessary and sufficient conditions for consensusability of linear
  multi-agent systems.
\newblock {\em {IEEE} Trans. Automat. Control}, 55(5):1263--1268, May 2010.

\bibitem{MesE10}
M.~Mesbahi and M.~Egerstedt.
\newblock {\em Graph Theoretic Methods in Multiagent Networks}.
\newblock Princeton Series in Applied Mathematics. Princeton University Press,
  Princeton, NJ, 2010.

\bibitem{morMliGB15}
P.~Mlinari\'{c}, S.~Grundel, and P.~Benner.
\newblock Efficient model order reduction for multi-agent systems using {QR}
  decomposition-based clustering.
\newblock In {\em 54th IEEE Conference on Decision and Control (CDC), Osaka,
  Japan}, pages 4794--4799, Dec. 2015.

\bibitem{MonCT15}
N.~Monshizadeh, M.~K. \c{C}amlibel, and H.~L. Trentelman.
\newblock Strong targeted controllability of dynamical networks.
\newblock In {\em 54th IEEE Conference on Decision and Control (CDC), Osaka,
  Japan}, pages 4782--4787. IEEE, Dec. 2015.

\bibitem{morMonTC14}
N.~Monshizadeh, H.~L. Trentelman, and M.~K. \c{C}amlibel.
\newblock Projection-based model reduction of multi-agent systems using graph
  partitions.
\newblock {\em {IEEE} Trans. Control Netw. Syst.}, 1(2):145--154, June 2014.

\bibitem{Mor05}
L.~Moreau.
\newblock Stability of multiagent systems with time-dependent communication
  links.
\newblock {\em {IEEE} Trans. Automat. Control}, 50(2):169--182, Feb. 2005.

\bibitem{New10}
M.~E.~J. Newman.
\newblock {\em Networks: An Introduction}.
\newblock Oxford University Press, Inc., New York, NY, USA, Mar. 2010.

\bibitem{Olf06}
R.~Olfati-Saber.
\newblock Flocking for multi-agent dynamic systems: algorithms and theory.
\newblock {\em {IEEE} Trans. Automat. Control}, 51(3):401--420, Mar. 2006.

\bibitem{OlfM03}
R.~Olfati-Saber and R.~M. Murray.
\newblock Consensus protocols for networks of dynamic agents.
\newblock In {\em Proceedings of the American Controls Conference}, volume~2,
  pages 951--956, June 2003.

\bibitem{Ran96}
A.~Rantzer.
\newblock On the {K}alman-{Y}akubovich-{P}opov lemma.
\newblock {\em Syst. Cont. Lett.}, 28(1):7--10, 1996.

\bibitem{morReiS07}
T.~Reis and T.~Stykel.
\newblock Stability analysis and model order reduction of coupled systems.
\newblock {\em Math. Comput. Model. Dyn. Syst.}, 13(5):413--436, 2007.

\bibitem{morSanM09}
H.~Sandberg and R.~M. Murray.
\newblock Model reduction of interconnected linear systems.
\newblock {\em Optim. Control Appl. Methods}, 30(3):225--245, 2009.

\bibitem{TreTM13}
H.~L. Trentelman, K.~Takaba, and N.~Monshizadeh.
\newblock Robust synchronization of uncertain linear multi-agent systems.
\newblock {\em {IEEE} Trans. Automat. Control}, 58(6):1511--1523, 2013.

\bibitem{morVanV08}
A.~Vandendorpe and P.~Van~Dooren.
\newblock {\em Model Reduction of Interconnected Systems}, pages 305--321.
\newblock Springer Berlin Heidelberg, Berlin, Heidelberg, 2008.

\end{thebibliography}

\end{document}